\documentclass[reqno]{amsart}
\usepackage{tkz-euclide}
\usepackage{hyperref}
\usepackage{amsmath}  
\usepackage{amsfonts}
\usepackage{amssymb} 
\usepackage{mathrsfs} 
\usepackage{graphicx}  
\usepackage{bm}
\usepackage{tikz} 
\usetikzlibrary{svg.path} 
\usepackage{amssymb}
\usepackage{caption}
\usepackage{subcaption}

\usetikzlibrary{calc} 
\usepackage{xcolor}  
\usetikzlibrary{arrows.meta} 
\usepackage{amsthm} 
\usepackage{float}
\usepackage{enumitem}
\usepackage[english]{babel}
\usepackage[font=footnotesize, labelfont=bf]{ caption}
\usepackage{ esint }
\usepackage{stackengine,scalerel,graphicx}
\stackMath

\definecolor{trueblue}{rgb}{0.0, 0.45, 0.81}
\definecolor{truegreen}{rgb}{0.13, 0.55, 0.13}

\newcommand{\EEE}{\color{black}}

\newcommand{\eps}{\varepsilon}

\theoremstyle{plain}
\begingroup
\newtheorem{theorem}{Theorem}[section]
\newtheorem{lemma}[theorem]{Lemma}

\newtheorem{remark}[theorem]{Remark}

\endgroup

\theoremstyle{definition}
\begingroup

\endgroup

\numberwithin{equation}{section}
\newcommand{\N}{\mathbb{N}}

\newcommand{\R}{\mathbb{R}}

\newcommand{\defas}{:=}

\usepackage[normalem]{ulem}

\begin{document}

\title[Existence of quasi-static crack evolution for atomistic systems]{Existence of quasi-static crack evolution for atomistic systems}

\author[R. Badal]{Rufat Badal} 
\address[Rufat Badal]{Department of Mathematics, Friedrich-Alexander Universit\"at Erlangen-N\"urnberg. Cauerstr.~11,
D-91058 Erlangen, Germany}
\email{rufat.badal@fau.de}

\author[M. Friedrich]{Manuel Friedrich} 
\address[Manuel Friedrich]{Department of Mathematics, Friedrich-Alexander Universit\"at Erlangen-N\"urnberg. Cauerstr.~11,
D-91058 Erlangen, Germany, \& Mathematics M\"{u}nster,  
University of M\"{u}nster, Einsteinstr.~62, D-48149 M\"{u}nster, Germany}
\email{manuel.friedrich@fau.de}

\author[J. Seutter]{Joscha Seutter} 
\address[Joscha Seutter]{Department of Mathematics, Friedrich-Alexander Universit\"at Erlangen-N\"urnberg. Cauerstr.~11,
D-91058 Erlangen, Germany \& Competence Unit for Scientific Computing (CSC), Friedrich-Alexander Universit\"at Erlangen-N\"urnberg (FAU)}
\email{joscha.seutter@fau.de}

\begin{abstract}
We consider  atomistic systems consisting  of interacting particles arranged in atomic lattices whose quasi-static evolution is driven by time-dependent boundary conditions. The interaction of the particles is modeled by classical interaction potentials where we implement a suitable irreversibility condition modeling the  breaking of atomic bonding. This leads to a delay differential equation depending on the complete history of the deformation at previous times. We prove existence of solutions and provide numerical tests for the prediction of quasi-static crack growth in particle systems.  
\end{abstract}

\subjclass[2010]{}

\maketitle

\section{Introduction}

The  theory of brittle fracture is based on the fundamental idea that the formation of cracks may be seen as the result of the competition between the elastic bulk energy of the body and the work needed to produce a new crack. Despite its long history in mechanics dating back to \cite{Griffith:20}, a rigorous mathematical formulation of the problem is much more recent: in the  variational model of quasi-static  crack \EEE growth   proposed in  \cite{frma98},  displacements and crack paths are determined from an energy minimization principle, reflecting suitably the three foundational ingredients of Griffith's theory: (1) a surface energy associated  to \EEE discontinuities of the deformation\EEE, (2) a propagation criterion for those surfaces, and (3) an irreversibility condition for the fracture process.


%

The mathematical well-posedness of the variational model  \cite{frma98} has been  the \EEE subject of several contributions over the last two decades, see e.g.\ \cite{DFT, DalMasoToader, Francfort-Larsen:2003, FriedrichSolombrino}. The proof of continuous-time evolutions is based on  approximation by time-discretized versions resulting from a sequence of minimization problems of so-called Griffith functionals, being variants of the   Mumford \& Shah  functional \cite{MS}. Functionals of this type comprising elastic bulk contributions for the sound regions of the body and surface terms along the crack have been studied extensively in the last years in the natural weak framework of special functions of bounded variation. We refer to \cite{BFM} for  a \EEE thorough overview of the variational approach to fracture. In particular, it has been shown that such continuum models can be identified
as effective limits derived from atomistic systems as the interatomic distance tends to zero \cite{Bach.Braides, Braides-Lew-Ortiz:06, FS151}. Asymptotic analysis of this kind is notoriously difficult, and is not based on  simply taking pointwise  limits \EEE but rather on a variational convergence, called $\Gamma$-convergence \cite{DalMaso:93}, which in this setting ensures convergence of minimizers of the atomistic system to the ones of the continuum problem.

Although  studies on  atomistic-to-continuum passages are flourishing, most results are static and the evolutionary nature of processes is mostly  neglected. Our goal is to advance this understanding by establishing a rigorous connection between atomistic and continuum models in brittle fracture for the prediction of quasi-static crack growth. In this contribution, we perform the first step in this direction by setting up a suitable quasi-static model on the atomistic level respecting the irreversibility condition of Griffith's theory. The passage to continuum models of quasi-static crack growth will be discussed in a forthcoming work \cite{FriedrichSeutter}.

We consider atomistic systems consisting  of interacting particles arranged in an atomic lattice occupying a bounded region  representing the reference domain of a material. The interaction of the atoms is modeled by  classical potentials in the frame of Molecular Mechanics, e.g.\ by Lennard-Jones potentials. Based on an energy minimization principle, the deformation of the particles is driven by time-dependent boundary conditions. This leads to an evolution of gradient-flow type. Whereas the study of interacting particles motion is classical, the novelty of our contribution lies in implementing an irreversibility condition along the evolution.  As long as the length of interactions does not exceed a certain  threshold, interactions are considered to be `elastic' and   can recover completely after removal of loading. Otherwise,  interactions are supposed to be `damaged', and the system is affected at the  present and all future times. Here, we suitably adapt the maximal opening memory variable  used in continuum models for cohesive crack evolutions \cite{Zanini, Giacomini} to the atomistic setting.

This results in a delay differential equation taking  the complete history of the deformation into account.  Our approach is based on the assumption that the boundary loading varies slowly in time compared to the elastic wave speed of the material. Therefore, inertial effects are neglected, and the delay differential equation  is a first order equation in time.  Our main result consists in proving existence of solutions, see Theorem \ref{th: main result},  where the strategy   relies on a time-incremental scheme, the so-called minimizing movement scheme for gradient flows \cite{AGS, DGMT, sant-am}. Similarly to the continuum model \cite{Francfort-Larsen:2003}, due to the presence of an irreversibility condition and the memory  effect, \EEE the main difficulty lies in  the passage to the limit as the time-discretization step $\tau$ tends to zero. In particular, we use an evolutionary model of rate type in order to obtain uniform convergence in time as $\tau \to 0$. This allows to pass to the limit in the memory variable, see  Remark \ref{rem: discussion} for some details.

Let us highlight that our  approach is rather simplistic and does not reflect the physical reality of fracture in brittle substances in its whole complexity. On the one hand, our modeling of elastic and damaged interactions is not justified by formation and dissociation  of chemical bonding  but inspired by a continuum mechanics perspective on brittle fracture. In this sense, we believe that our model may be relevant as a discretization of quasi-static crack evolution from continuum theory. On the other hand, physically relevant interatomic potentials describing the behavior of brittle materials need to encompass multi-body contributions, at least three-body potentials \cite{brenner, ters, still}. Following the approach of  several   studies on variational atomistic fracture \cite{Alicandro-Focardi-Gelli:2000, Braides-DalMaso-Garroni:1999, Braides-Gelli:2002-2, braid-gell, FS153}, we focus here on configurational energies consisting purely of pair interactions.  Nevertheless, we present a possible extension of our model to three-body potentials in Subsection \ref{sec: variants}.

\EEE

Our analysis is at the basis of a forthcoming work \cite{FriedrichSeutter} where we study the evolution of the model in the passage to continuum  problems, i.e., when the interatomic distance tends to zero. Indeed, in the static case, $\Gamma$-limits of suitable atomistic systems can be identified as free discontinuity functionals of Griffith-type  \cite{Braides-Lew-Ortiz:06, FS151} featuring a linearly elastic bulk energy and surface contributions proportional to the length of the crack. For these continuum models, existence results for quasi-static crack evolutions have been  derived \EEE in \cite{Francfort-Larsen:2003, FriedrichSolombrino}.  We \EEE will identify evolutions of this form as limits of the solutions to the atomistic delay differential equations in the atomistic-to-continuum limit.

The paper is organized as follows. In Section \ref{sec:setting} we introduce the model and present the main results. Section \ref{sec: examples} is devoted to some explicit examples of possible atomistic systems, and provides several numerical experiments of our model in one  and \EEE two dimensions. Eventually, Section \ref{sec: proofs} contains the  proofs of our results. \EEE  We close the introduction with some notation. In the following, we denote by $|\cdot|$ the standard Euclidean norm in $\mathbb{R}^{m}$ for $m\in \N$. For the scalar product of two vectors $v,w \in \mathbb{R}^{m}$, we write $v \cdot w$. We let $a \vee b:=\max\{a,b\}$ for $a,b\in \mathbb{R}$. As usual, generic constants in proofs are denoted by $C$ and may change from line to line.

\section{The model and main results} \label{sec:setting}

\subsection{Deformations and interaction energy}\label{sec: defo}

We describe the atomistic deformation of a material occupying an open, bounded  set \EEE  $\Omega \subset \R^d$ in dimension $d \in \lbrace 1, 2, 3 \rbrace$. By $\mathcal{L} \subset \R^d$ we denote an arbitrary but fixed \emph{lattice}, and let $\mathcal{L}(\Omega) = \Omega \cap \mathcal{L}$ be the positions of the \emph{atoms} (or particles) in the specimen in the reference configuration. The deformation of the particles is described through  a map \EEE $y \colon \mathcal{L}(\Omega) \to \R^n$, where for each $x \in \mathcal{L}(\Omega)$ the vector $y(x) \in \R^n$ indicates the position of $x$ in the deformed configuration. (The typical choice is $n=d$.) 

We model the interaction of the atoms by  classical potentials in the frame of Molecular Mechanics \cite{Molecular,  Lewars}: given a deformation $y \colon \mathcal{L}(\Omega) \to \R^n$,  we assign a \emph{phenomenological energy} $\mathcal{E}(y)$ to the deformation by 
\begin{align}\label{Energy}
\mathcal{E}(y) = \frac{1}{2}\sum_{\underset{x \neq x'}{x, x' \in \mathcal{L}(\Omega)}} W_{x,x'}\big( |y(x) - y(x')| \big) ,
\end{align}
where $W_{x,x'} \colon [0,\infty) \to  \R$ denotes a smooth  \emph{two-body interaction potential}, with $W_{x,x'} = W_{x',x}$. The factor $\frac{1}{2}$  accounts for the fact that every contribution is counted twice in the sum.  It is also possible that the sum ranges only over a subset of all pairs by simply setting $W_{x,x'} \equiv 0$ for certain interactions. Whenever $W_{x,x'} \neq 0$, the potential \EEE  is supposed to be repulsive for small distances and attractive for long distances, a classical choice being  a potential of \emph{Lennard-Jones-type}. More specifically, we assume that $W_{x,x'}$ satisfies

\begin{enumerate}[label=(\roman*)]
\item \label{W_zero} $\lim_{r \searrow 0 } W_{x,x'}(r) = \infty$;
\item \label{W_min} $W_{x,x'}(r)$ is minimal if and only if $r = |x-x'|$ with $W_{x,x'}(|x-x'|) <0$;
\item \label{W_decr_incr} $W_{x,x'}(r)$ is decreasing for $r < |x-x'|$ and increasing for $r > |x-x'|$;
\item \label{W_inf} $\lim_{r \to +\infty} W_{x,x'}(r) = \lim_{r \to +\infty} \frac{{\rm d} }{{\rm d} r}W_{x,x'}(r) = \lim_{r \to +\infty} \frac{{\rm d}^2 }{{\rm d}^2 r}W_{x,x'}(r) =  0$.
\end{enumerate}

The repulsion effect  \ref{W_zero} describes the Pauli repulsion at short distances of the interacting particles due to overlapping electron orbitals, and \ref{W_min}-\ref{W_inf} correspond to attraction at long ranged interactions vanishing at infinity. Some toy examples of lattices and possible configurational energies are given below in Section \ref{sec: examples}. Let us mention that we restrict our model to pair interactions equilibrated in the reference configuration (see Assumption \ref{W_min}) for mere simplicity. The model could be generalized to pre-stressed interactions, and multi-body, non-local, and multiscale energies could be considered, see e.g.\ \cite{Bach.Braides} for a very general framework. We include, however, non-homogeneous interactions as all potentials $W_{x,x'}$ can be different. A possible extension featuring three-body potentials is addressed in Subsection~\ref{sec: variants} below.  

By $\eps: = \min \{|x-x'| \colon x, x' \in \mathcal L(\Omega), \, x \neq x' \}$ we denote the minimal   \emph{interatomic distance}.  A huge body of literature deals with effective descriptions of atomistic energies of the form \eqref{Energy} when the interatomic distance $\eps$ tends to $0$. This passage from atomistic to continuum models is nontrivial even for simple energies and might lead to very different asymptotic behavior, see \cite{Blanc-LeBris-Lions:2007}  for taking pointwise limits and \cite{Braides-Gelli} for an overview of a variational perspective. For interaction densities of the above form, it has been shown in \cite{Braides-Lew-Ortiz:06, FS151} that effective continuum energies lead to so-called \emph{Griffith energy functionals} featuring both bulk and surface terms. In the present contribution, the number of atoms and $\eps$ are fixed. The study of  an evolutionary model as $\eps \to 0$ is deferred to a forthcoming paper \cite{FriedrichSeutter}.

Following the discussion in \cite{Schmidt:2009, FS151}, we impose Dirichlet boundary condition on a part $\partial_D \Omega \subset \partial \Omega$ of the boundary. More precisely, consider a neighborhood $ U \EEE \subset \Omega$ of $\partial_D \Omega$ such  that $\mathcal{L}_{D}(\Omega) := \mathcal{L}(\Omega) \cap  U \EEE \neq \emptyset$. F\EEE or a function $g \colon \mathcal{L}_{D}(\Omega) \to  \R^n\EEE$, we define  the \emph{admissible deformations} by \EEE
$$\mathcal{A}(g) := \big\{ y\colon \mathcal{L}(\Omega ) \to \R^n \colon \,  y(x) = g(x) \  \ \text{ for all } x \in \mathcal{L}_{D}(\Omega);\, \mathcal{E}(y)<\infty \big\} .$$ 
Instead of (or additional to) Dirichlet boundary conditions, one could also consider body or boundary forces. However, as traction  tests with body forces are ill-posed for Griffith functionals in a continuum variational formulation, we prefer to focus on boundary value problems also on the atomistic level. For a  thorough discussion on the comparison of soft and hard devices in the context of the variational formulation of Griffith's fracture, we refer the reader to \cite[Section 3]{BFM}. 

\subsection{Evolutionary model: Irreversibility condition and memory variable}\label{sec: evo model}

Based on the energy \eqref{Energy}, we will now introduce an evolutionary model on a  process \EEE  time interval  $[0,T]$, where the evolution will be driven by time dependent boundary conditions $g \colon [0,T] \times \mathcal{L}_{D}(\Omega) \to \R^n$. The key feature in the modeling is to implement an irreversibility condition along the fracture process  for each pair interaction \EEE  of \EEE $(x,x')$, once the   length  $|y(x) - y(x')|$ \EEE of the interaction exceeds the elastic regime. As long as   $|y(x) - y(x')|$ does not  surpass \EEE a threshold $R^1_{x,x'} = R^1_{x',x}$ satisfying $R^1_{x,x'} > |x-x'|$, the pair interaction can recover completely after removal of loading. By contrast, if $|y(x) - y(x')|$ exceeds $R^1_{x,x'}$ at some earlier time, the interaction is supposed to be `damaged', and the system should be affected at the present time, i.e., the complete history of the deformation undergone by the atomistic system should be reflected at present time. This corresponds to a model of cohesive type energy \'a la Barenblatt \cite{barenblatt}.

To describe this, we fix $t \in [0,T]$,  a deformation $y$ at time $t$, \EEE and suppose that a finite or infinite  family of deformations $(y_s)_{s <t}$ at previous times is given. 
The interaction  of \EEE $(x,x')$ at time $t$ is considered to be elastic if $|y_s(x) - y_s(x')| \le R^1_{x,x'}$ for all $s < t$. Otherwise, we adopt as memory variable the \emph{maximal opening} as implemented, e.g.\ in the continuum models \cite{Zanini, Giacomini}, see also \cite[Section 5.2]{BFM}. Writing
\begin{align}\label{eq: memory}
M_{x,x'}((y_s)_{s<t}) = \sup_{s < t}   |y_s(x') - y_s(x)|   \vee |x-x'| \EEE 
\end{align}
for the memory variable, the energy contribution  with memory \EEE  reads as 
\begin{align}\label{eq: memory1}
 W_{x,x'} \Big(   M_{x,x'}((y_s)_{s<t}) \vee  |y(x') - y(x)|  \Big).   
\end{align}
Other choices, such as that of a cumulative increment, could also be made, but are not discussed in this paper. Note that by implementing the memory variable, the repulsion effect for small distances (Assumption  \ref{W_zero}\EEE) is lost and \eqref{eq: memory1} does not describe correctly the Pauli repulsion at short distances of the interacting particles. Therefore,  for the energy contribution of a  damaged interaction  we rather use \EEE
\begin{align*}
W_{x,x'} (|y(x') - y(x)|)  \vee  W_{x,x'} \Big(   M_{x,x'}((y_s)_{s<t}) \vee  |y(x') - y(x)|  \Big).   
\end{align*}
Clearly,  by Assumption \ref{W_decr_incr} \EEE this   is equal to 
\begin{align}\label{eq: memory1.5}
W_{x,x'} (|y(x') - y(x)|)  \vee  W_{x,x'} \Big(   M_{x,x'}((y_s)_{s<t})  \Big).   
\end{align}
 This notion is rather simplistic and one drawback lies in the sharp transition between elastic and inelastic behavior. To remedy this, we introduce a second threshold $R^2_{x,x'} > R^1_{x,x'}$   and assume that  in the transition zone  $(R^1_{x,x'}, R^2_{x,x'})$ \EEE the interaction can partially recover, i.e.,  the energy contribution still partly depends on the current deformation. Let $\varphi_{x,x'} \colon [0,\infty) \to [0,1]$ be continuous and  increasing \EEE with $\varphi_{x,x'}(r) = 0 $ for $r \le R^1_{x,x'}$ and  $\varphi_{x,x'}(r) = 1 $ for $r \ge R^2_{x,x'}$. Define the energy contribution as 
\begin{align}\label{eq: memory2}
{E}_{x,x'}(y; (y_s)_{s<t}) :=  &\big(1- \varphi_{x,x'}\big(   M_{x,x'}((y_s)_{s<t}) \big)  \big)     W_{x,x'} \big( |y(x') - y(x)| \big) \\ &+  \varphi_{x,x'}\big(   M_{x,x'}((y_s)_{s<t}) \big) \Big( W_{x,x'} (|y(x') - y(x)|) \vee W_{x,x'}(M_{x,x'}((y_s)_{s<t}))     \Big) .     \notag   
\end{align}
Note that formally this would correspond to \eqref{eq: memory1.5} for the (not admissible) choice  $\varphi_{x,x'}(r) = 0 $ for $r \le R^1_{x,x'}$ and  $\varphi_{x,x'}(r) = 1 $ for $r  >   R^1_{x,x'}$. 
In particular,  \eqref{eq: memory2} \EEE implies
$${E}_{x,x'}(y; (y_s)_{s<t})  = W_{x,x'} \big( |y(x') - y(x)| \big) \quad \text{ if } M_{x,x'}((y_s)_{s<t})  \le R^1_{x,x'} $$
$$ {E}_{x,x'}(y; (y_s)_{s<t})  =  W_{x,x'} (|y(x') - y(x)|) \vee W_{x,x'}(M_{x,x'}((y_s)_{s<t}))       \quad \text{ if } M_{x,x'}((y_s)_{s<t})  \ge R^2_{x,x'}. $$
We adopt the choice \eqref{eq: memory2} for a continuous function interpolating between $0$ and $1$ also for technical reasons, see the discussion in  Remark \ref{rem: discussion}.
Summarizing, given the history $(y_s)_{s<t}$, the energy of  an admissible configuration $y \in \mathcal{A}(g(t))$ is given by  
\begin{align}\label{eq: main energy}
\mathcal{E}(y; (y_s)_{s<t}) = \frac{1}{2}\sum_{\underset{x \neq x'}{x, x' \in \mathcal{L}(\Omega)}} {E}_{x,x'}(y; (y_s)_{s<t}). 
\end{align}

\subsection{Quasi-static evolution of atomistic systems with damageable interactions}\label{sec: result}

We now present the main result of the paper on a quasi-static evolution of atomistic systems with the irreversibility condition introduced above.  We introduce some further notation.  Let  $N$ be the cardinality of $\mathcal{L}(\Omega)$  and let $\bar{N}$ be the cardinality of $(\mathcal{L}(\Omega) \setminus \mathcal{L}_D(\Omega))$, respectively.  \EEE For convenience, \EEE  from now on we will denote $y \colon \mathcal{L}(\Omega) \to \R^n$ as a single vector $y \in \R^{Nn}$ via 
$$y = (y(x_1), \ldots y(x_N)) \in \R^{Nn},$$ 
where $(x_1,\ldots,x_N)$ is a fixed numbering of $\mathcal{L}(\Omega)$ with $\lbrace x_{\bar{N}+1}, \ldots, x_N\rbrace = \mathcal{L}_D(\Omega)$.  In this sense,  $\mathcal{E}(\cdot; (y_s)_{s<t})$ in  \eqref{eq: main energy} is a function defined on $\R^{Nn}$ with variables $y_1,\ldots, y_N$.  We also regard boundary conditions $g \colon [0,T] \times \mathcal{L}_{D}(\Omega) \to \R^n$ as vectors $g(t) \in \R^{Nn}$ for each $t \in [0,T]$ by setting $g(t,x) = 0$ for all $x \in \mathcal{L}(\Omega) \setminus \mathcal{L}_D(\Omega)$. \EEE

 As it is customary in both  existence proofs  and in numerical approximation schemes, we start with a time-discretization of the problem. Fix a discrete time step size $\tau >0$, which for simplicity satisfies  $T/\tau \in \N$.  We write $t_k^{\tau}=k\tau$ for $k \in \lbrace 0, \ldots,  T/\tau \rbrace$, \EEE and we also define $g^\tau_k = g(t^\tau_k)$. Suppose that an initial configuration $y_0 \in \mathcal{A}(g^\tau_0)$ is given. Fix a time $t^\tau_k$ and suppose that the deformations  $(y^\tau_j)_{j < k}$ at the previous times $t^\tau_j$, $j < k$, have already been found.  Then, the deformation at time $t^\tau_k$, denoted by $y^\tau_k  \in \mathcal{A}(g^\tau_k)$, is given as the minimizer of
\begin{align}\label{eq: MM} \min_{y \in \mathcal{A}(g^\tau_k)}   \Big( \mathcal{E}(y; (y^\tau_j)_{j<k})  +   \frac{\nu}{2 \tau}  |y - y_{k-1}^\tau |^2   \Big). 
\end{align}  
Here, $|\cdot|$ denotes the  Euclidean norm in $\R^{Nn}$,  where \EEE $\nu>0$  is \EEE a \emph{dissipation coefficient}.   We remark that \eqref{eq: MM} corresponds to the \emph{minimizing movement scheme} for gradient flows  of \EEE the energy $\mathcal{E}(\cdot; (y^\tau_j)_{j<k})$ with respect to the $L^2$-norm. At the end of the section, see \eqref{eq: the new scheme}, we present a variant of  \eqref{eq: MM} which corresponds to a model in  the Kelvin-Voigt's rheology.  Note that an essential part of our problem is  the fact \EEE that the energy depends on the complete history, i.e., on the previous deformations  $(y^\tau_j)_{j<k}$, which is usually not present in incremental schemes of the form  \eqref{eq: MM}. Let us mention that this  relates \EEE  to a discrete-time formulation where, in contrast to   common variational \EEE approaches to fracture in the realm of rate-independent systems, in each step    we do not consider absolute minimizers of the energy, but determine local minimizers by penalizing the $L^2$-distance between the approximate solutions at two consecutive times. This approach itself is not new  and \EEE has been adopted, e.g. in \cite{DalMasoToader}.

Note that the existence of a minimizer is readily guaranteed since the problem is finite dimensional, the potential is continuous, and the energy is coercive.   Moreover, one can also check that for $\tau$ sufficiently small the problem is strictly convex, and therefore the minimizer is unique.

We now introduce piecewise constant and piecewise affine interpolations  
by  $y_{\tau}(t):=y_{k+1}^{\tau}$  for $t\in (t_k^{\tau},t_{k+1}^{\tau}]$ and 
$\hat{y}_{\tau}(t):=y_k^{\tau}+(t-k\tau)v_{\tau}(t)$ for $t\in (t_k^{\tau},t_{k+1}^{\tau}]$, respectively, where  $v_{\tau}(t):=\frac{1}{\tau}(y^\tau_{k+1}-y^\tau_k) $.
Our goal is to pass to the limit $\tau\to 0$, and to prove that $y_{\tau}$ converges uniformly to a limit evolution $y$, which  satisfies a delay differential equation taking the irreversibility of the fracture process into account. We start with the compactness result.

\begin{theorem}[Compactness]\label{th: compactness}
Assume that $g \in H^1([0,T]; \R^{Nn})$. 
Suppose that an initial datum $y_0 \in \mathcal{A}(g(0))$ is given, and let $y_{\tau}$ and $\hat{y}_{\tau}$ be constructed as above. Then, \EEE there exists $y \in H^1([0,T]; \R^{Nn})$ with $y(t) \in \mathcal{A}(g(t))$ for all $t \in [0,T]$ such that $y_{\tau}$ and   $\hat{y}_{\tau}$ converge uniformly to $y$ on $[0,T]$, as well as   $\partial_t\hat{y}_{\tau}  \rightharpoonup \partial_t y $ weakly in $L^2([0,T]; \R^{Nn} )$.          
\end{theorem}

We are now ready to present the main result of the paper.

\begin{theorem}[Quasi-static evolution for damageable pair interactions]\label{th: main result}
Given an initial value  $y_0 \in \mathcal{A}(g(0))$, denote by  $y \in H^1([0,T]; \R^{Nd})$ the function identified in Theorem \ref{th: compactness}. Then, $y$ satisfies the delay differential equation
\begin{align}  \label{eq: maineq}
\nabla \mathcal{E}(y(t); (y(s))_{s<t})=-\nu\partial_t y(t) \quad \text{on $\mathcal{L}(\Omega) \setminus \mathcal{L}_{D}(\Omega)$ for a.e.\ $t \in [0,T]$}.
\end{align}
\end{theorem}
 Here, $\nabla$ is taken with respect to the variables $y_1,\ldots, y_{\bar{N}}$, i.e., both sides in \eqref{eq: maineq} are vectors in $\R^{\bar{N}n}$.  More explicitly, \eqref{eq: maineq}  can be written as \EEE 
$$\nabla_{y_i} \mathcal{E}(y(t); (y(s))_{s<t})=-\nu\partial_t y(t,x_i) \quad \text{for $i=1,\ldots,\bar{N}$ for a.e.\ $t \in [0,T]$}. $$
 \EEE Although the potentials $W_{x,x'}$ are smooth, $\mathcal{E}(\cdot; (y(s))_{s<t})$ is not necessarily differentiable due to the presence of the memory variable. Therefore, strictly speaking, $\nabla$  has to be interpreted as the element of the subdifferential of $\mathcal{E}$ (in the sense of $\lambda$-convex functions) with minimal norm,  denoted by $\partial^\circ \mathcal{E}$. We refer to \EEE  \eqref{eq: subdiff1}--\eqref{eq: subdiff2} for details.

Let us highlight that the resulting equation is not an ordinary differential equation of gradient flow type but a delay differential equation as the energy  at time $t$ depends on the complete history $(y(s))_{s<t}$.  

\subsection{Variants and possible extensions}\label{sec: variants}

We close this section by discussing  alternative modeling assumptions and possible extensions. \EEE Let us start by presenting a variant of the model where the $L^2$-dissipation is replaced by a dissipation depending on a (discrete) strain rate. This corresponds to a model in the Kelvin-Voigt's rheology.  For simplicity, we restrict our modeling here to the one-dimensional case  ($d=n=1$) \EEE as proofs are  more involved in higher dimensions. Specifically, we denote the reference configuration by $x_1 =  \eps, \EEE \ldots, x_N = \eps N $ with $x_i - x_{i-1} = \eps$ for $i \in \lbrace 2, \ldots,N\rbrace$, and denote deformations $y$ by $y = (y(x_1), \ldots, y(x_N))  \in \R^N\EEE$. Let $\mathcal{L}_{D}(\Omega)  = \lbrace x_1,x_N \rbrace$,  where for convenience we use a different labeling of $\mathcal{L}_{D}(\Omega)$ compared to Subsection \ref{sec: result}. \EEE We introduce the \emph{discrete gradient} by 
\begin{align}\label{eq: dg}
\bar{\nabla} y = \Big(\frac{y(x_i)- y(x_{i-1})}{\eps}\Big)_{i=2,\ldots,N} \in \R^{N-1}.
\end{align}
We also introduce the vector $\bar{\rm D} y \in \R^{N-2}$ defined by 
\begin{equation}\label{def: VektorD}
  (\bar{\rm D} y)_i =    \frac{2y(x_i)- y(x_{i-1}) - y(x_{i+1})}{\eps}, i \in \lbrace 2, \ldots, N-1\rbrace. \end{equation}
We replace the time-incremental scheme described in \eqref{eq: MM} 
by
\begin{align}\label{eq: the new scheme}
  \min_{y \in \mathcal{A}(g^\tau_k)}   \Big( \mathcal{E}(y; (y^\tau_j)_{j<k})  +   \frac{\nu}{2 \tau}  |\bar{\nabla} y - \bar{\nabla} y_{k-1}^\tau |^2   \Big). 
\end{align} 
In the limit of vanishing time steps $\tau \to 0$, this leads  to the delay differential equation
\begin{align}\label{eq: new model}
\partial^\circ \mathcal{E}(y(t); (y(s))_{s<t})=-\nu \bar{\rm D} \partial_t y(t) \quad \text{ on $\lbrace x_2,\ldots, x_{N-1} \rbrace$ \EEE for a.e.\ $t \in [0,T]$},
\end{align}
where $\partial^\circ \mathcal{E}$ is explained in \eqref{eq: subdiff1}--\eqref{eq: subdiff2} and is taken with respect to  $(y_2,\ldots, y_{N-1})$.  The proof is similar to the one of Theorem \ref{th: main result} and the necessary adaptions are  given \EEE in Remark \ref{rem: adapt}. Let us mention that a more realistic modeling assumption would be that $\nu$ is different for each contribution and depends on the memory variable  $M_{x,x'}$, with $\nu(M_{x,x'})$ decreasing in $M_{x,x'}$. \\

The model developed so far was restricted to pair interactions  for the sake of simplicity. However, for a realistic description of  brittle materials also  multi-body interactions should be taken into account. Here, we present one possible simple extension incorporating three-body interactions.
Let  $x,x',x''\in \mathcal{L}(\Omega)$ be three distinct neighboring atoms in the reference configuration with $|x'-x|=\varepsilon$ and $|x'-x''|=\varepsilon$. Consider the angle between the two edges $(x',x)$ and $(x',x'')$ under the deformation $y$, given  by \[\theta_{x,x',x''}(y):=\arccos \big\langle y(x)-y(x'), y(x'')-y(x') \big\rangle \in [0,\pi]\,.\] 
Let   $W_{x,x',x''}\colon[0,\pi] \to \R$ be a continuous three-body interaction potential, see \cite{brenner, ters, still} for possible choices. We then assume that the three-body energy contribution of $x,x',x''\in \mathcal{L}(\Omega)$ 
has the form
\[ E_{x,x',x''}(y, (y_s)_{s<t}):=W_{x,x',x''}\big(\theta_{x,x',x''}(y)\big)\Big[1-\big(\varphi\big(M_{x,x'}((y_s)_{s<t})\big)\vee \varphi\big(M_{x',x''}((y_s)_{s<t})\big)\big)\Big]\,.\]
This reflects the assumption that the three-body-interaction is only active as long as the two interactions $(x,x')$ and $(x',x'')$ are not 'damaged'.
The total energy in \eqref{eq: main energy} then changes to 
\begin{align}\label{eq: alternative energy}
    \mathcal{E}(y; (y_s)_{s<t}) = \frac{1}{2}\sum_{\underset{x \neq x'}{x, x' \in \mathcal{L}(\Omega)}} {E}_{x,x'}(y; (y_s)_{s<t}) +  \sum_{\underset{|x-x'|=|x''-x'|=\varepsilon}{x, x',x'' \in \mathcal{L}(\Omega), x \neq x''}}E_{x,x',x''}(y, (y_s)_{s<t})\,.
\end{align}    
Note that Theorem \ref{th: main result} stated above still holds for this extended model, as we briefly indicate after the proof of Theorem \ref{th: main result}, see Remark \ref{rem: new}.

\EEE

\section{Examples and numerical experiments}\label{sec: examples}

The proposed model for discrete crack evolution captures a variety of possible atomic interactions. In preparation for our numerical experiments, we discuss here simple models in 1D and 2D.

 The easiest case consists in a one-dimensional set-up described at the end of Subsection \ref{sec: result}  with  nearest neighbor interactions.  If $|x-x'|\neq \eps$  it holds \EEE $W_{x, x'} \equiv 0$, and for $|x-x'|=\eps$ the interaction takes the form  
\begin{align}\label{eq: NNI}
  W_{x, x'}(|y(x) - y(x')|) := W(|y(x) - y(x')|), \qquad \text{where }
  W(r) :=   4\left( \Big(\frac{\eps}{\sqrt[6]{2} r}\Big)^{12} - \Big(\frac{\eps}{\sqrt[6]{2} r}\Big)^6 \right).
\end{align}
We also choose some parameters $\eps < R^1 < R^2$ for the memory variable independently of $(x,x')$. \EEE

In two dimensions, an important mathematical model is given by a triangular lattice $\mathcal{L}=  \eps \EEE A\mathbb{Z}^2=  \eps \EEE \begin{pmatrix}
  1 & \frac{1}{2} \\ 
  0 & \frac{\sqrt{3}}{2}
\end{pmatrix} \mathbb{Z}^2$ 
and  nearest neighbor \EEE interactions of the form \eqref{eq: NNI}.  This model was considered in \cite{FS151}.  Here, the deformations $y$ are assumed to locally preserve their orientation, i.e.,  affine interpolation should have nonnegative determinant on each triangle of the lattice $\mathcal{L}$, cf.\ \cite{Schmidt:2009}.  Additionally, one could also add next-to-nearest neighbors to the energy, i.e., the interaction of points $x$, $x'$ of the lattice opposite to the same edge (with distance $\sqrt{3} \eps$). In this case, we add 
\begin{align}\label{eq: NNIXX}
W_{x, x'}(|y(x) - x(x')|) = \eta W(\tfrac{1}{\sqrt{3}}|y(x) - y(x')|) 
\end{align}
to the energy for each such pair, where $\eta >0$ and $W$ is  given by \eqref{eq: NNI}. \EEE 
\EEE

%

%

\begin{figure}[h]
  \centering
  \begin{subfigure}[b]{0.49\textwidth}
      \centering
      \includegraphics[width=\textwidth]{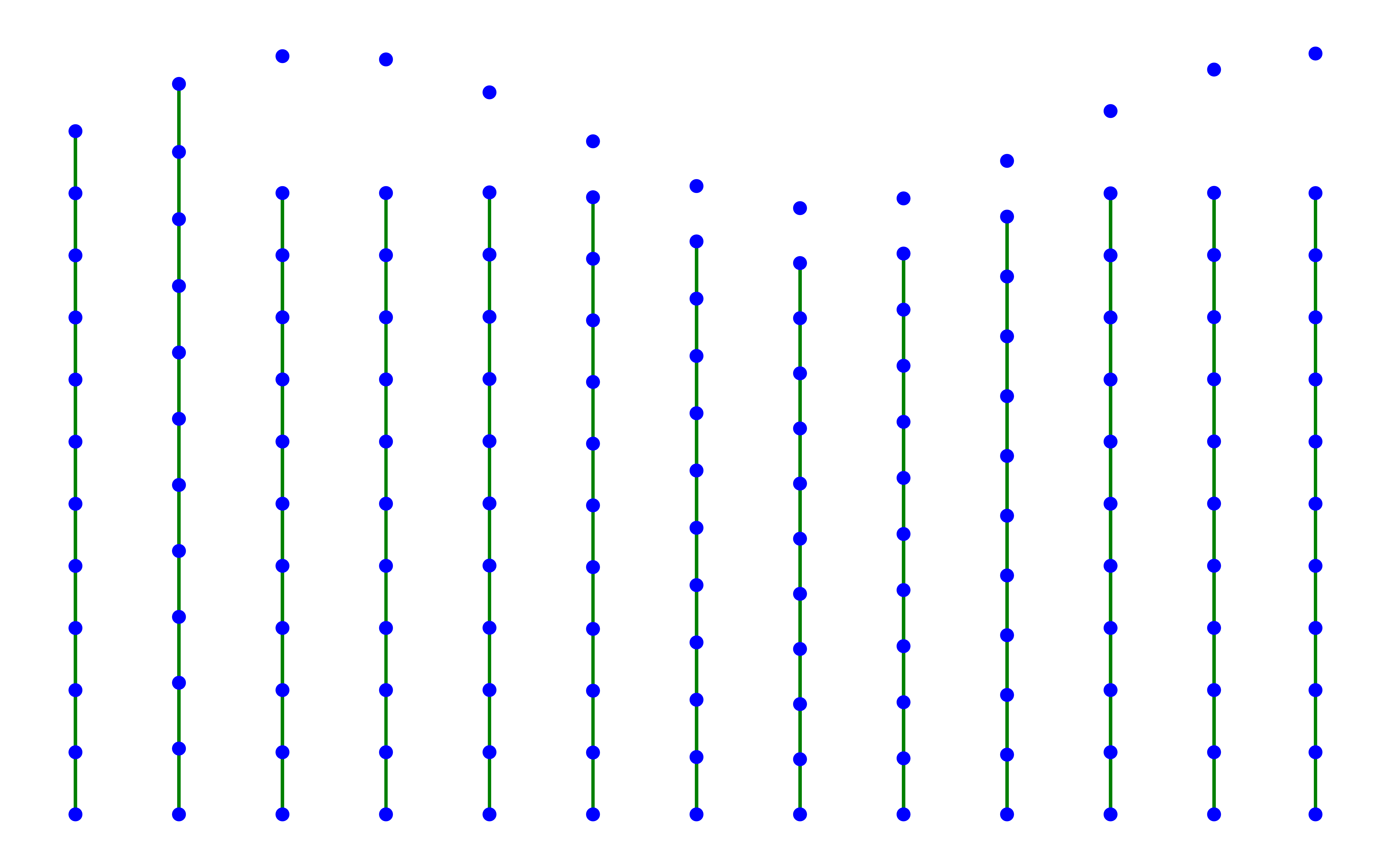}
      \caption{$L^2$}
      \label{fig:line_l2}
  \end{subfigure}
  \hfill
  \begin{subfigure}[b]{0.49\textwidth}
      \centering
      \includegraphics[width=\textwidth]{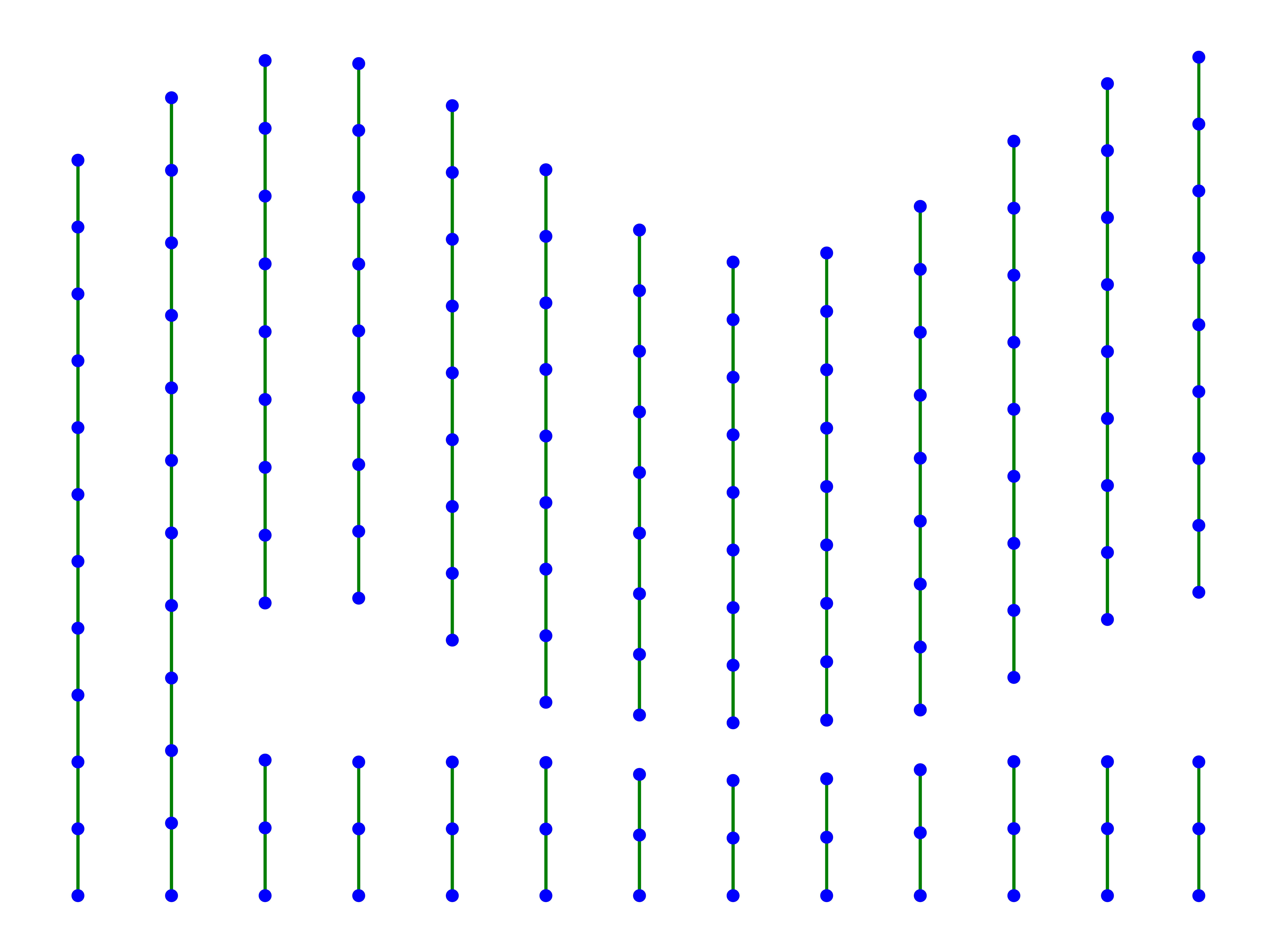}
      \caption{Kelvin-Voigt}
      \label{line_kv}
  \end{subfigure}
     \caption{Evolution of a chain of atoms. The blue dots in each column describe the deformed configuration of atoms for one specific step of the scheme with time progressing towards the right. A green line is drawn between each nearest-neighbor pair until their distance goes beyond the elastic threshold.}
     \label{fig:line}
\end{figure}

We now describe numerical experiments for the 1D and the 2D case. In dimension one, we consider a chain of $N$ equally spaced atoms with interatomic distance $\eps$.
To ease numerical computations, we only consider interaction energies \eqref{eq: main energy} of nearest-neighbor type  as given in \eqref{eq: NNI}. \EEE We set $R^1 =  1.2 \EEE\eps$, and for simplicity $R^2 = R^1$. As dissipation, we consider both cases mentioned above,   i.e., the $L^2$-distance and the Kelvin-Voigt-dissipation.  We set $\tau = 1/60$ and $\nu = 0.1$. \EEE    The numerical computation of each step of the minimizing movements schemes in \eqref{eq: MM} and \eqref{eq: the new scheme} is performed via an interior point optimizer  \cite{madnlp} accessed through the JuMP modeling language  \cite{jump}.

In Figure \ref{fig:line}, one can see  13 \EEE steps of both schemes.  The evolution starts with the chain of atoms with distance $\eps$ to each other.
The boundary condition fixes one of the endpoints while driving the other endpoint in sinusoidal motion stretching and compressing the chain in the process.
In both cases,  the chain first extends elastically and at some point \EEE one of the bonds between neighboring atoms breaks.
While in the case of Kelvin-Voigt, by symmetry, any of the bonds is equally likely to break (in the specific case above it is the $3$-th bond),   the $L^2$-dissipation always favors the very last bond  since already for small loading this bond is  stretched the most. After compression, in the second turn of extension, no elastic stretching can be observed  since one bond is already broken. \EEE

  \begin{figure}[h]
  \centering
  \begin{subfigure}[b]{0.298\textwidth}
      \centering
      \includegraphics[width=\textwidth]{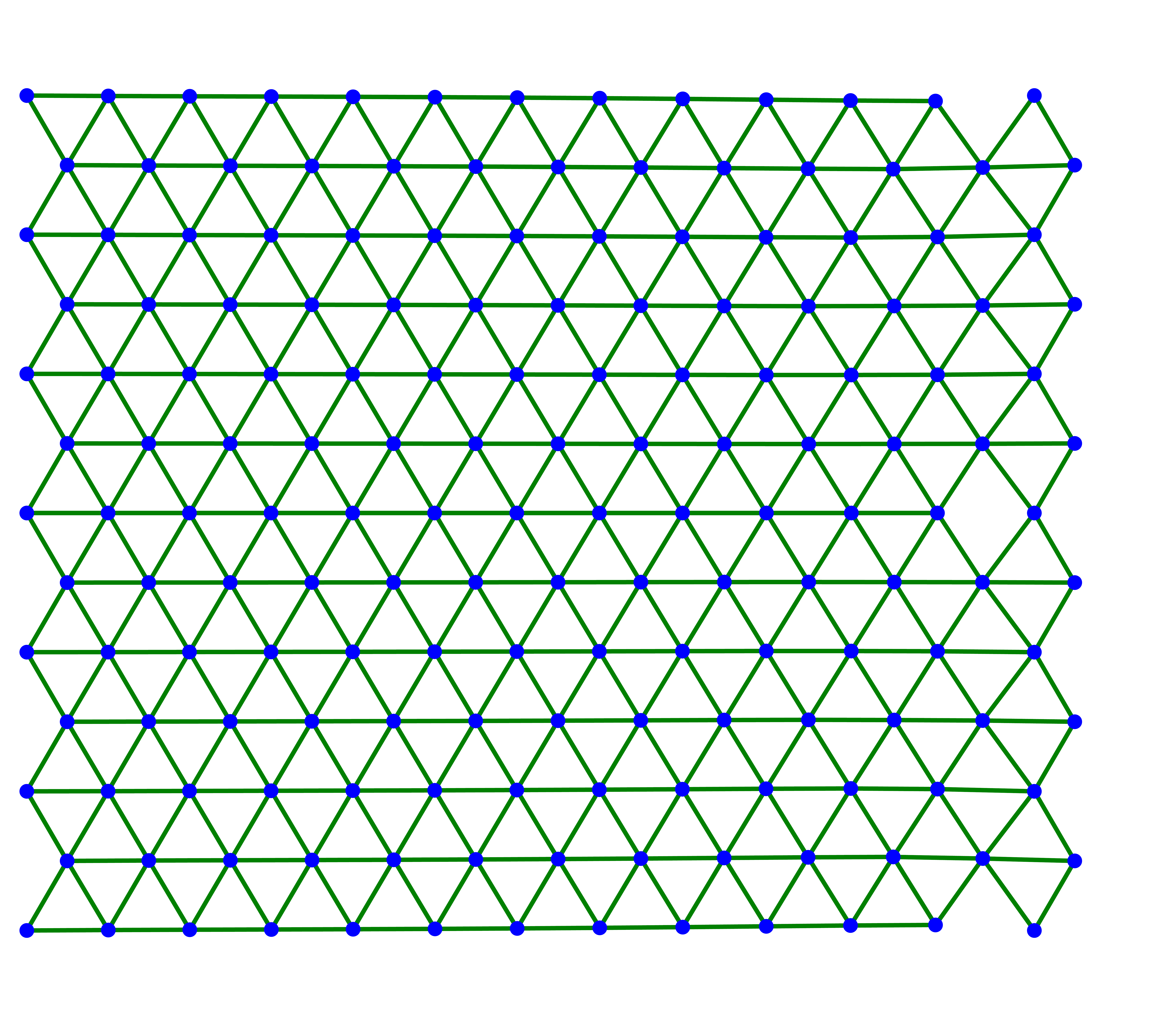}
      \caption{Step 11}
      \label{fig:square_l2_1}
  \end{subfigure}
  \hfill
  \begin{subfigure}[b]{0.298\textwidth}
      \centering
      \includegraphics[width=\textwidth]{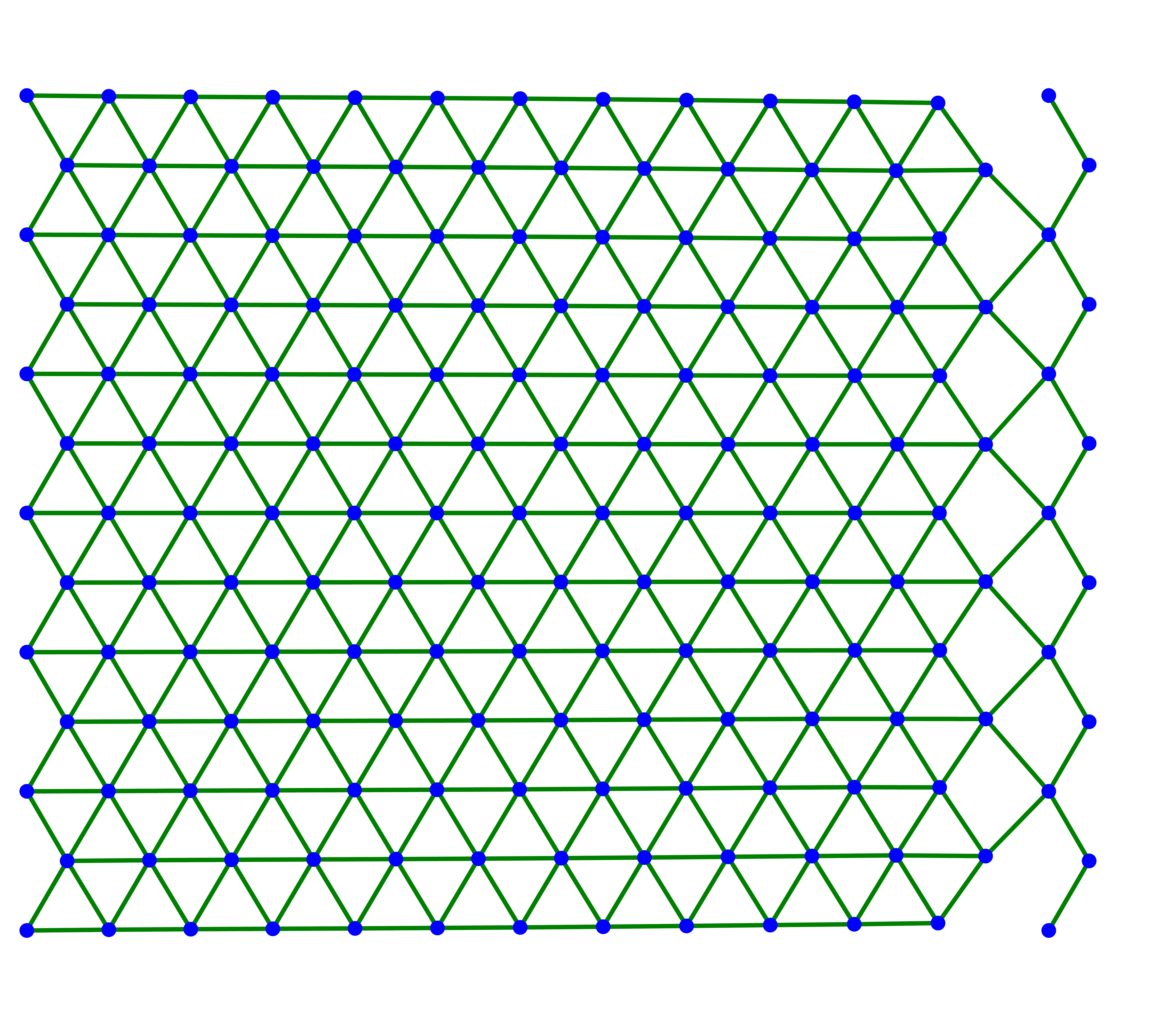}
      \caption{Step 15}
      \label{fig:square_l2_2}
  \end{subfigure}
  \hfill
  \begin{subfigure}[b]{0.298\textwidth}
      \centering
      \includegraphics[width=\textwidth]{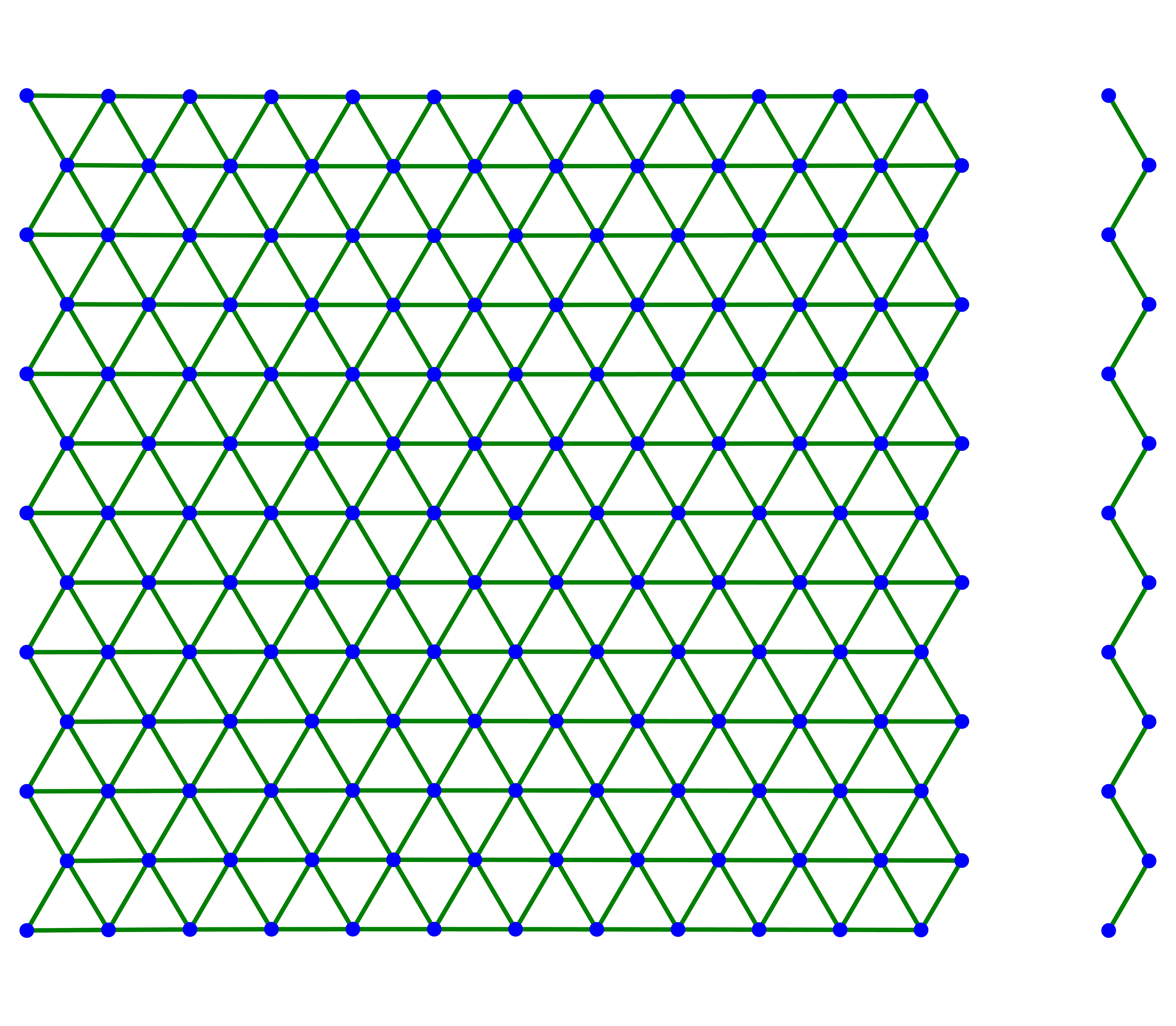}
      \caption{Step 45}
      \label{fig:square_l2_3}
  \end{subfigure} \\
  \begin{subfigure}[b]{0.298\textwidth}
      \centering
      \includegraphics[width=\textwidth]{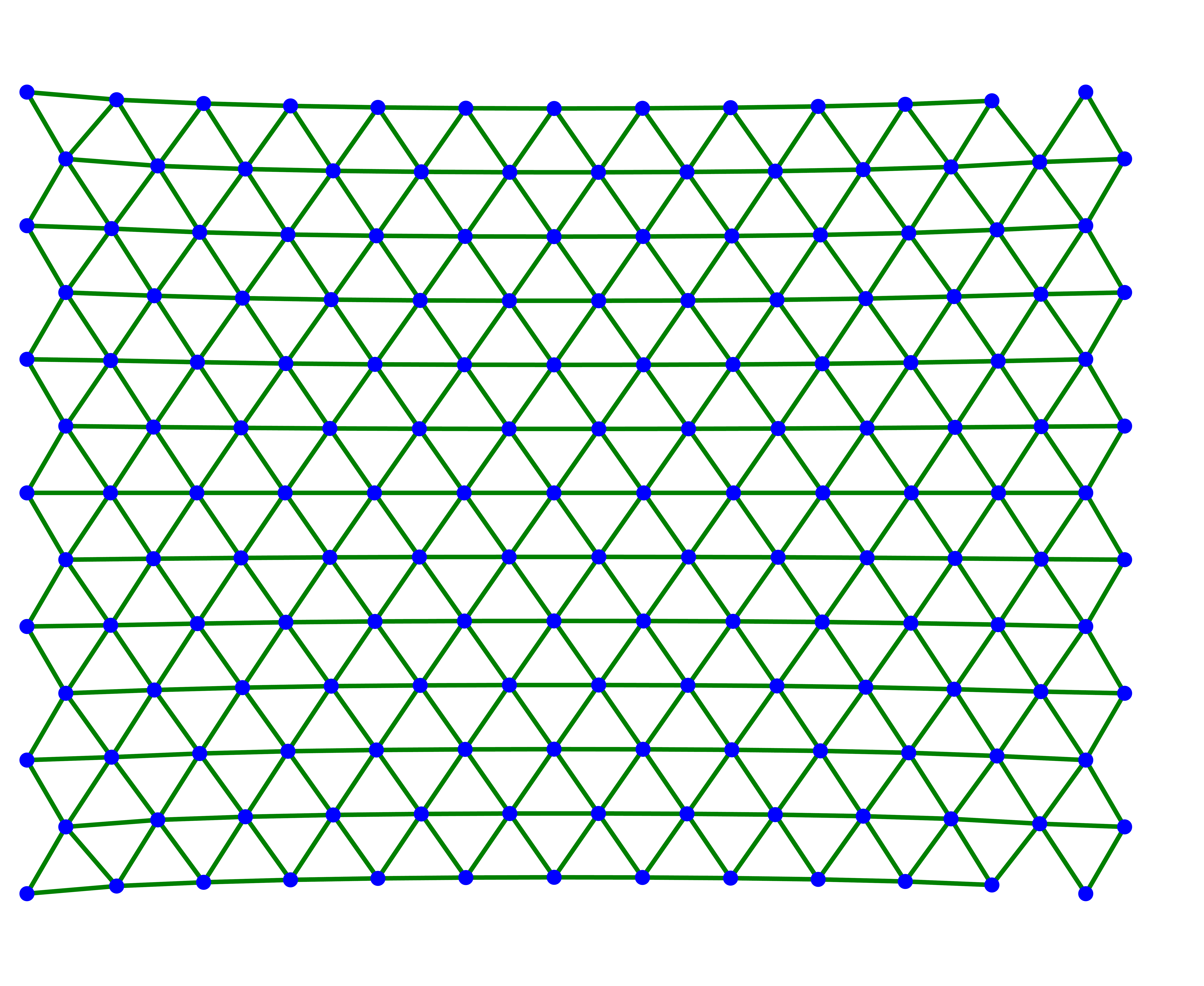}
      \caption{Step 16}
      \label{fig:square_kv_1}
  \end{subfigure}
  \hfill
  \begin{subfigure}[b]{0.298\textwidth}
      \centering
      \includegraphics[width=\textwidth]{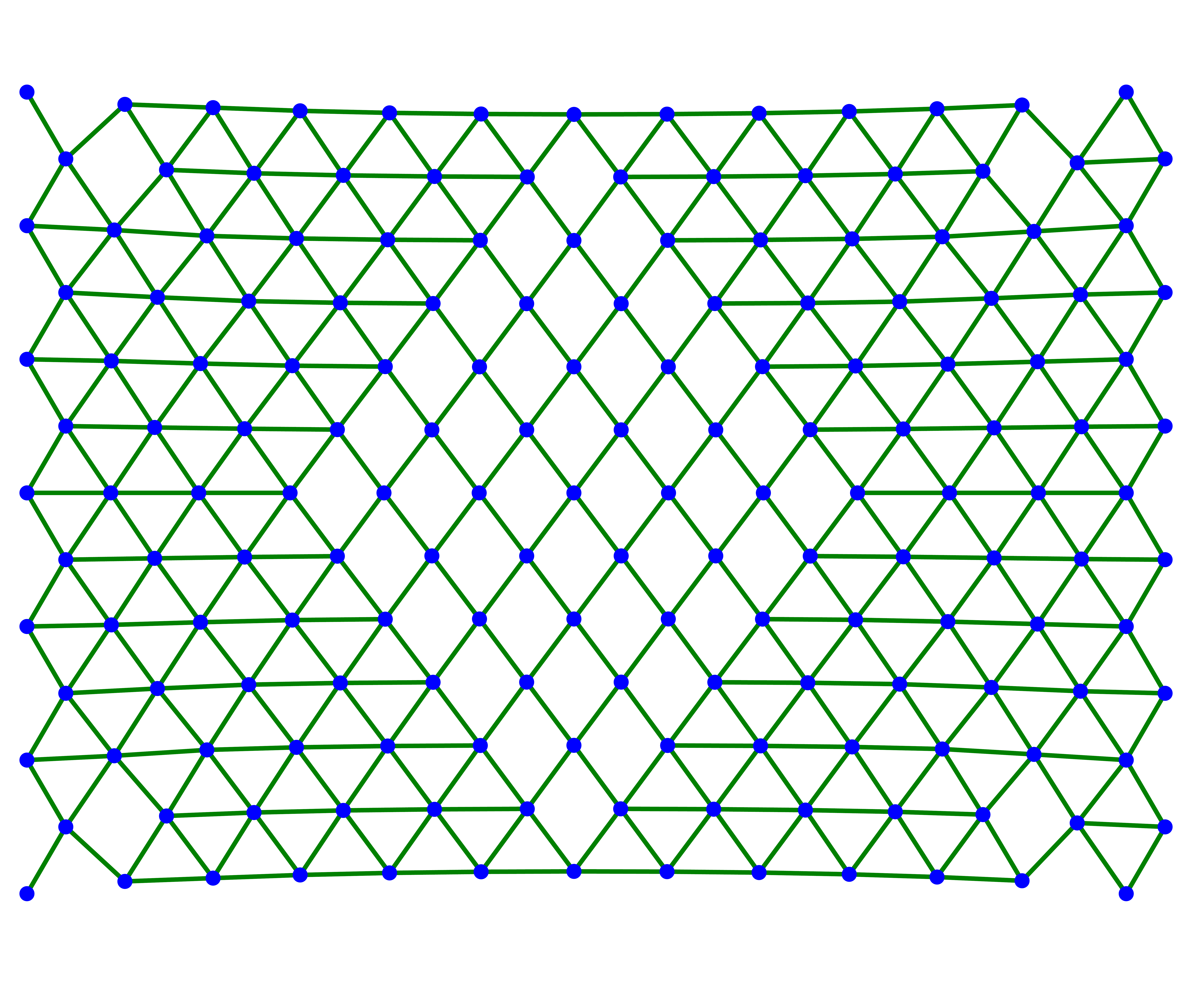}
      \caption{Step 24}
      \label{fig:square_kv_2}
  \end{subfigure}
  \hfill
  \begin{subfigure}[b]{0.298\textwidth}
      \centering
      \includegraphics[width=\textwidth]{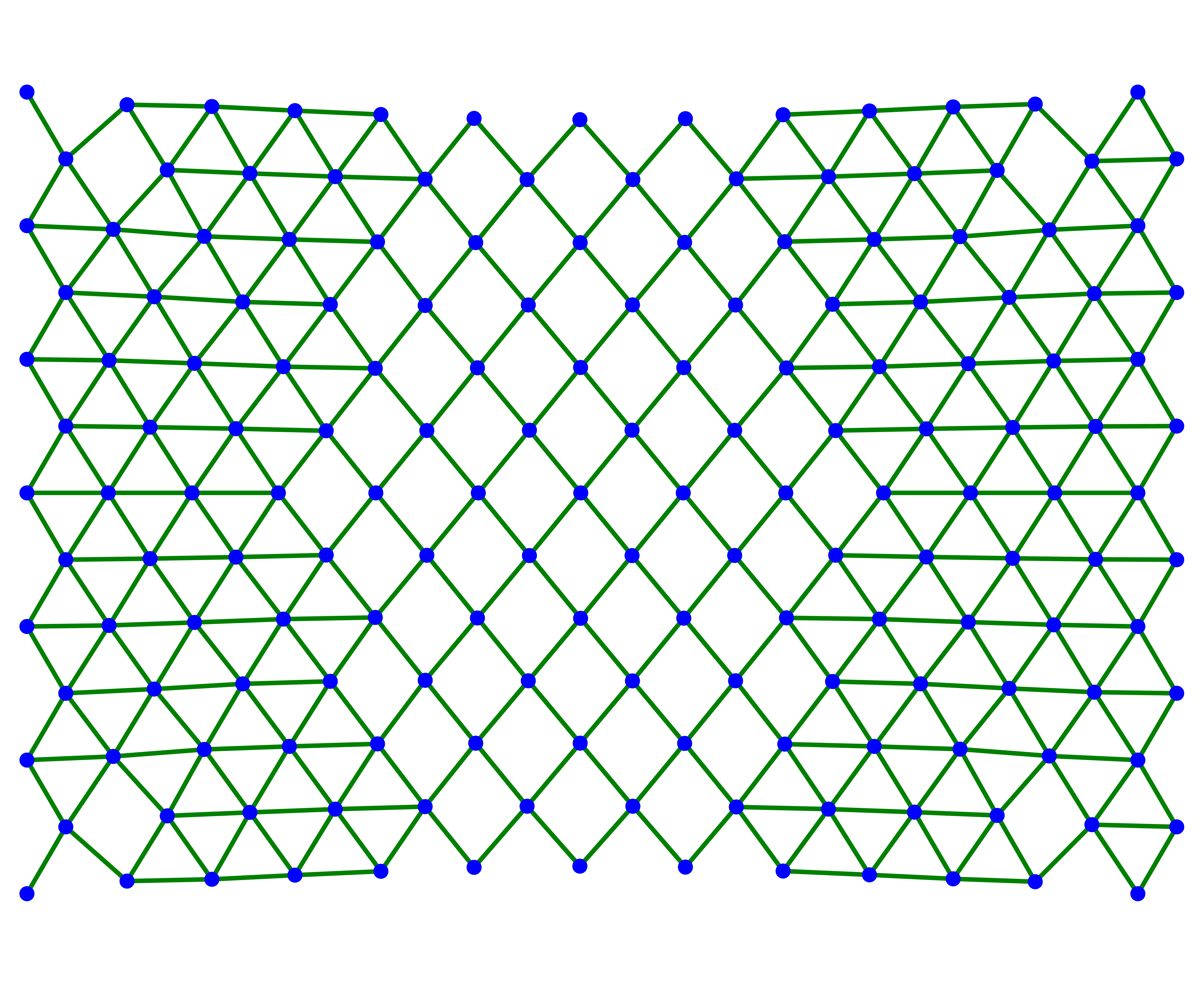}
      \caption{Step 30}
      \label{fig:square_kv_3}
  \end{subfigure}
    \caption{Evolution of a triangular lattice for large $\nu = 1$.  Green lines again correspond to bonds in the elastic regime. \EEE }
    \label{fig:square_extr}
\end{figure}
\begin{figure}
    \label{fig: stress_strain}
\end{figure}

\begin{figure}[h!]
  \centering
  \begin{subfigure}[b]{0.298\textwidth}
      \centering
      \includegraphics[width=\textwidth]{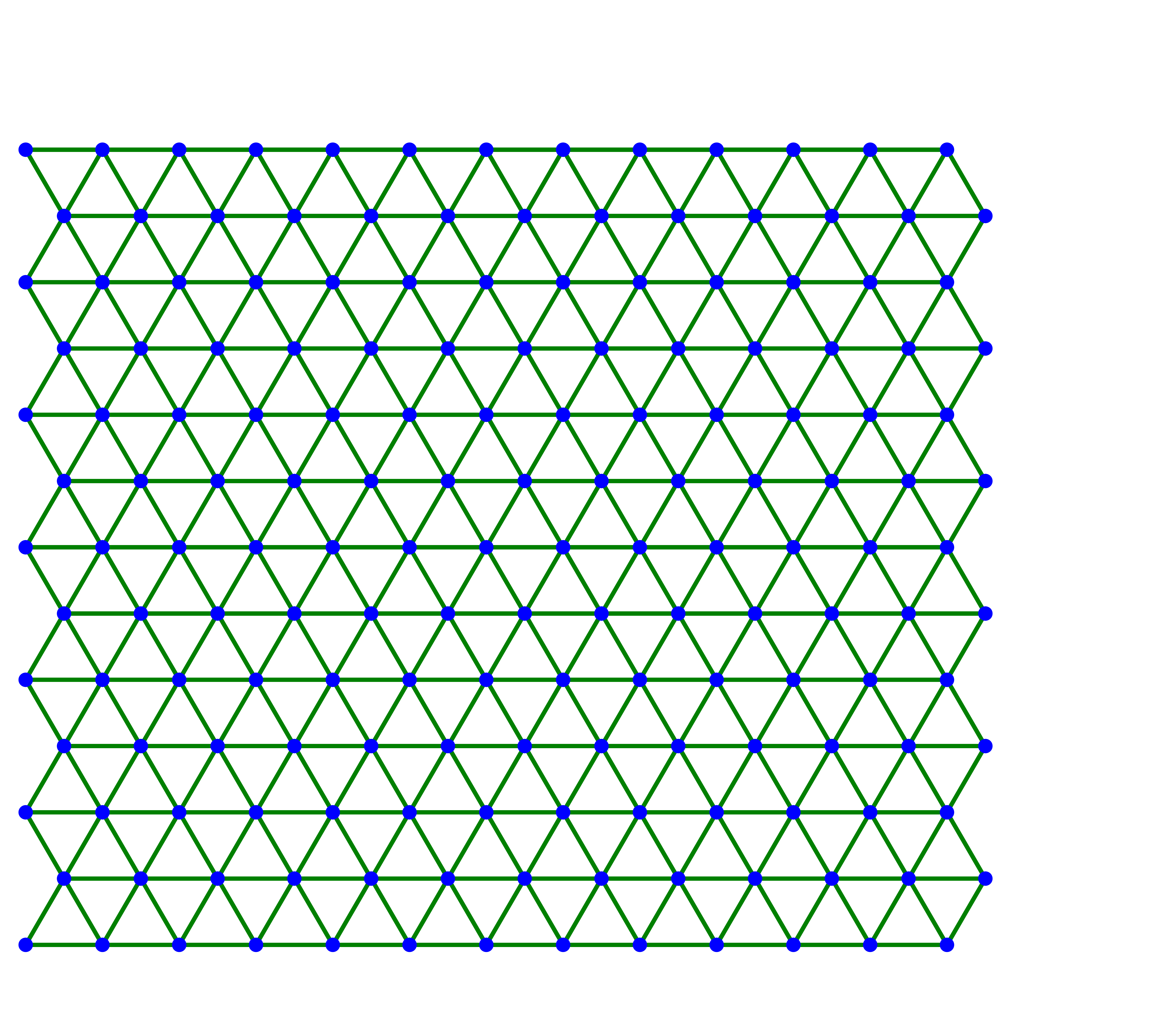}
      \caption{Step 1}
      \label{fig:square_diag_1}
  \end{subfigure}
  \begin{subfigure}[b]{0.298\textwidth}
      \centering
      \includegraphics[width=\textwidth]{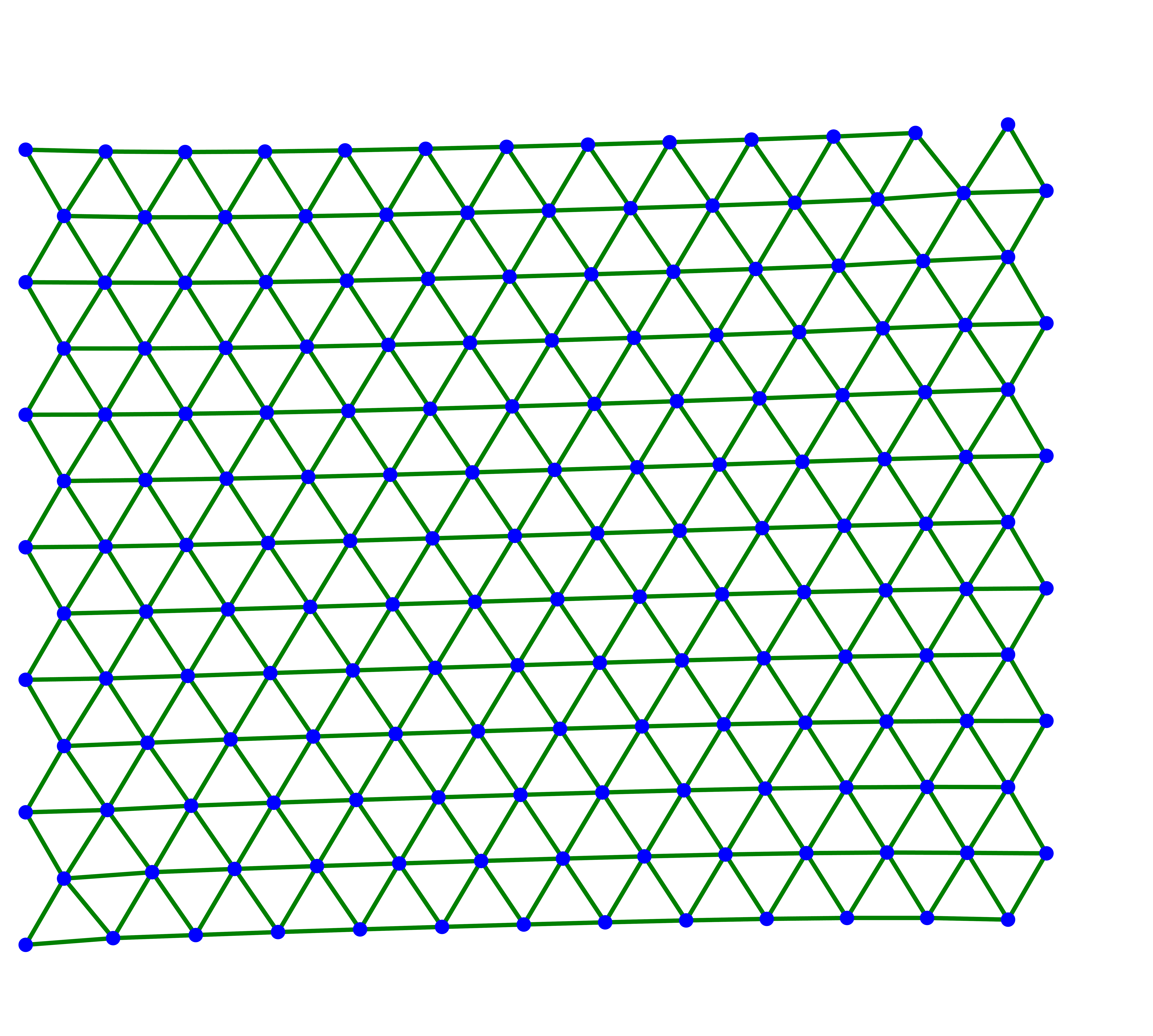}
      \caption{Step 18}
      \label{fig:square_diag_2}
  \end{subfigure}
  \hfill
  \begin{subfigure}[b]{0.298\textwidth}
      \centering
      \includegraphics[width=\textwidth]{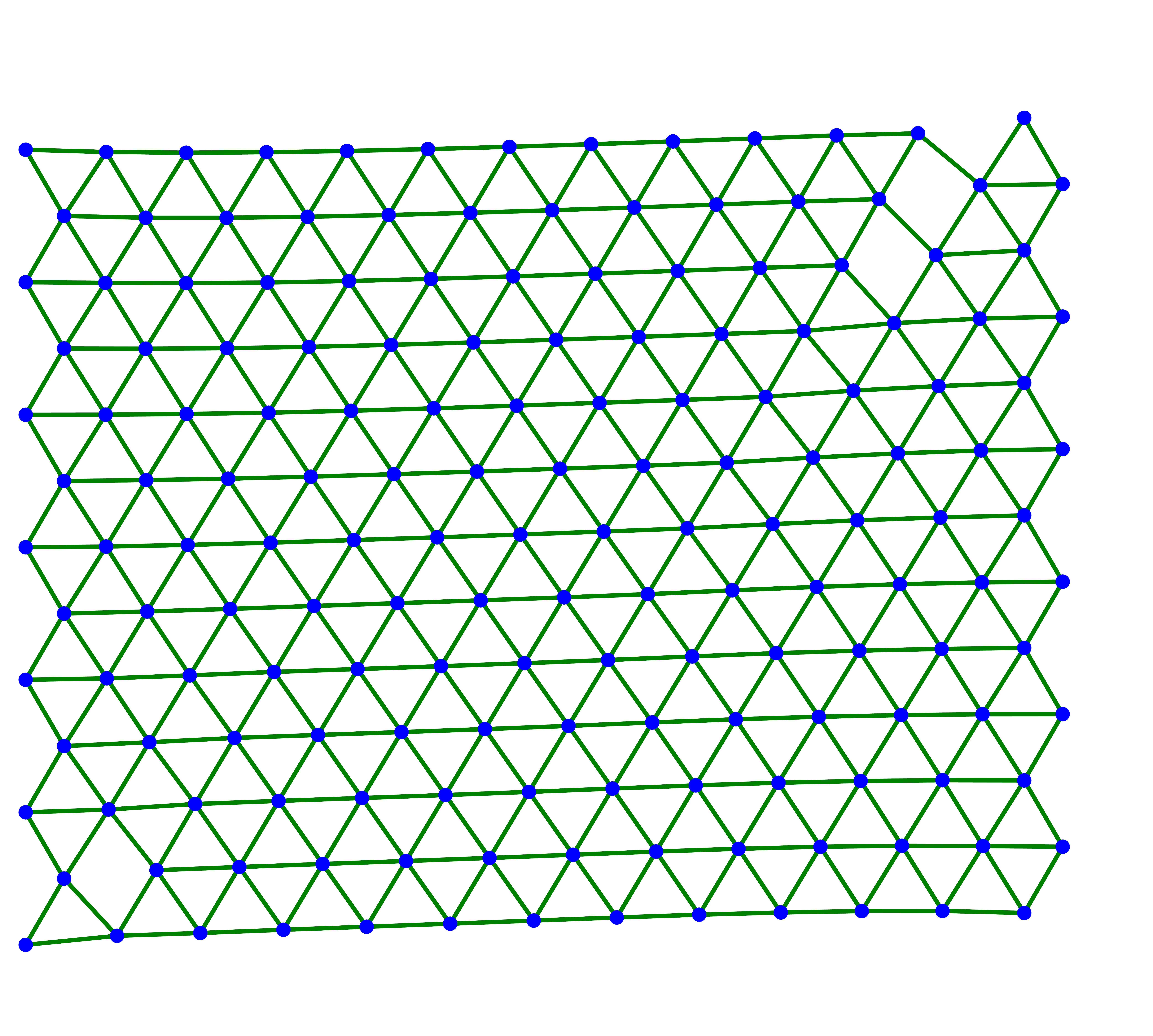}
      \caption{Step 23}
      \label{fig:square_diag_3}
  \end{subfigure} \\
  \begin{subfigure}[b]{0.298\textwidth}
      \centering
      \includegraphics[width=\textwidth]{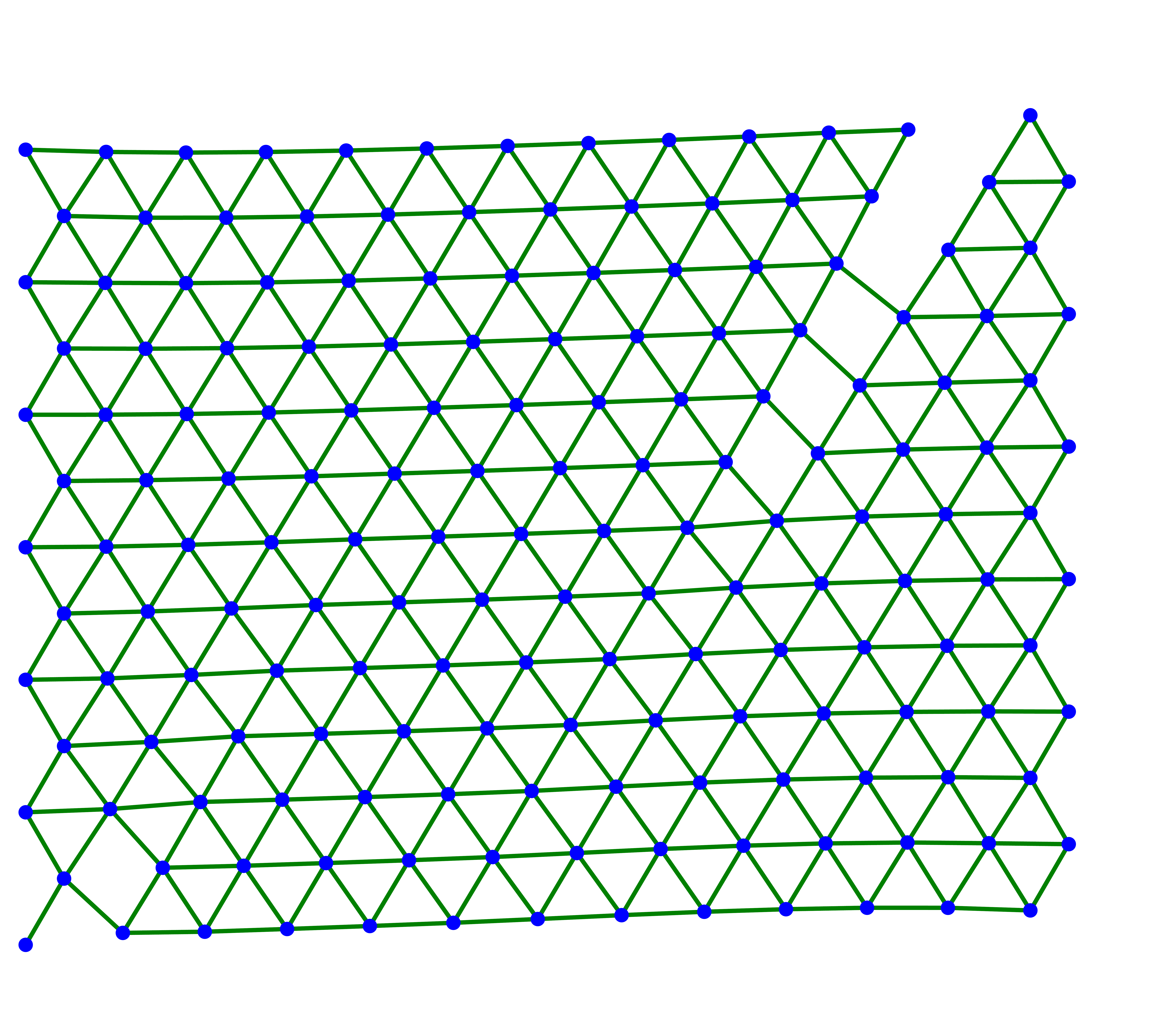}
      \caption{Step 25}
      \label{fig:square_diag_4}
  \end{subfigure}
  \hfill
  \begin{subfigure}[b]{0.298\textwidth}
      \centering
      \includegraphics[width=\textwidth]{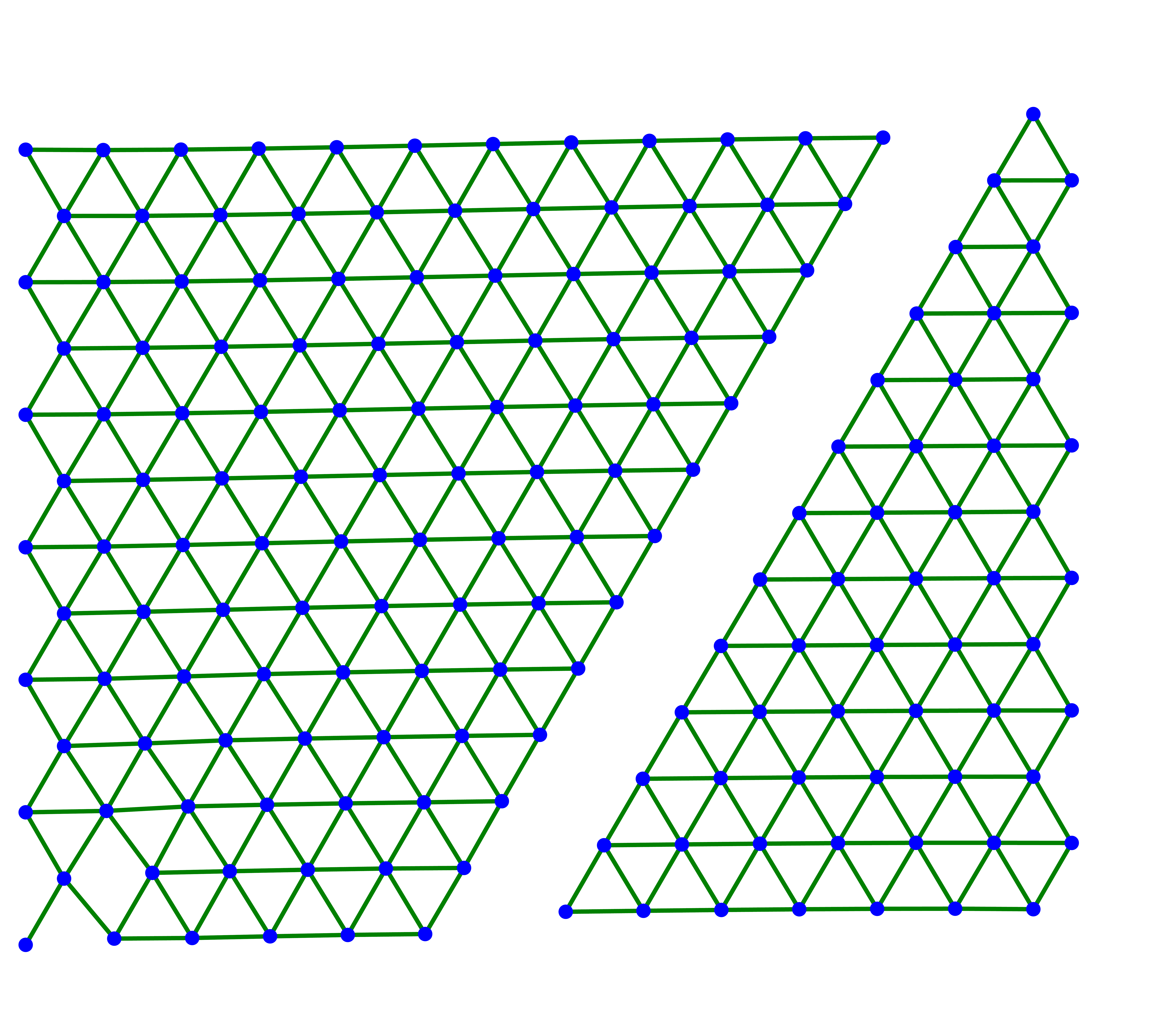}
      \caption{Step 26}
      \label{fig:square_diag_5}
  \end{subfigure}
  \hfill
  \begin{subfigure}[b]{0.298\textwidth}
      \centering
      \includegraphics[width=\textwidth]{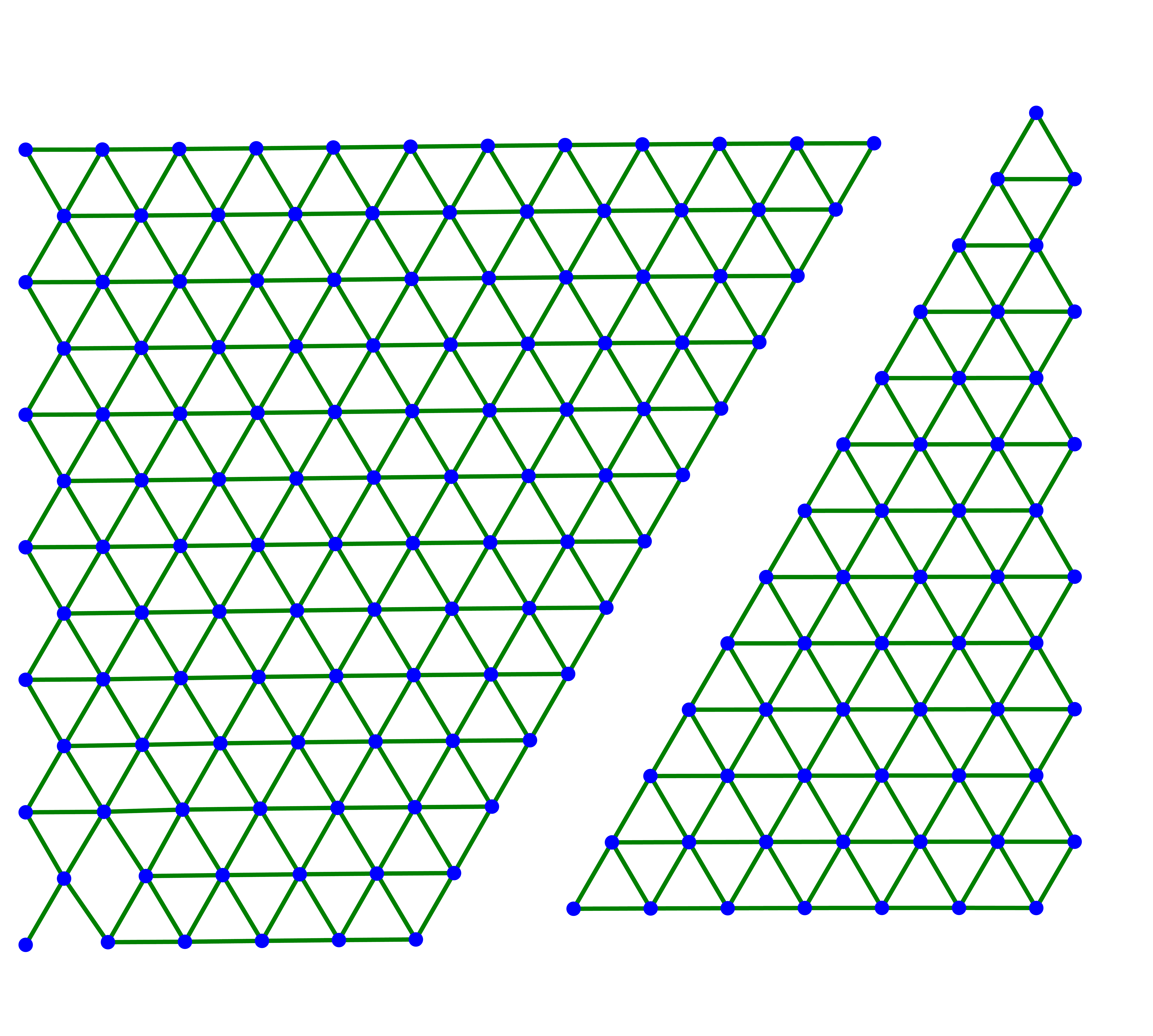}
      \caption{Step 27}
      \label{fig:square_diag_6}
  \end{subfigure} \\
  \begin{subfigure}[b]{0.298\textwidth}
      \centering
      \includegraphics[width=\textwidth]{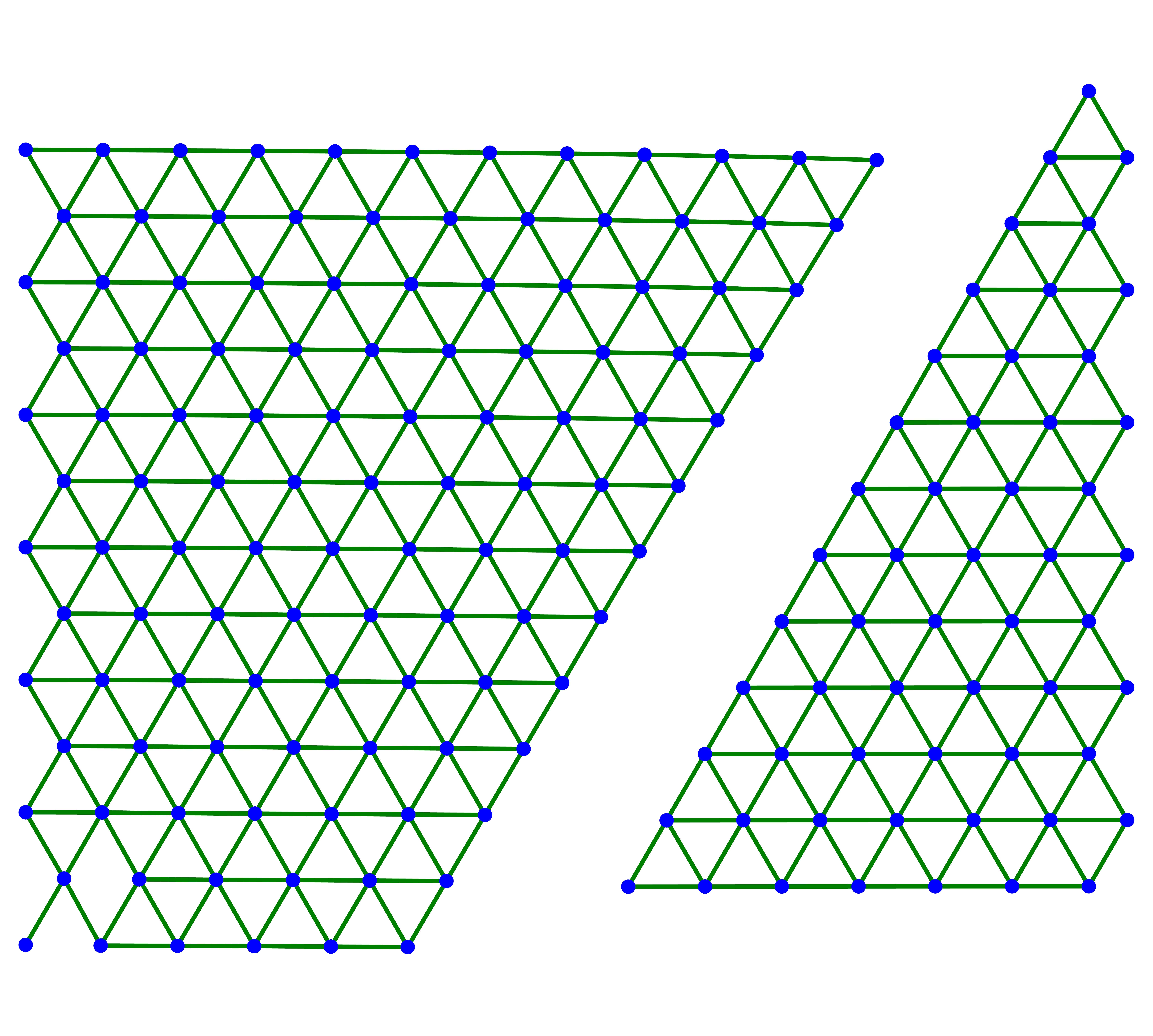}
      \caption{Step 60}
      \label{fig:square_diag_7}
  \end{subfigure}
  \hfill
  \begin{subfigure}[b]{0.298\textwidth}
      \centering
      \includegraphics[width=\textwidth]{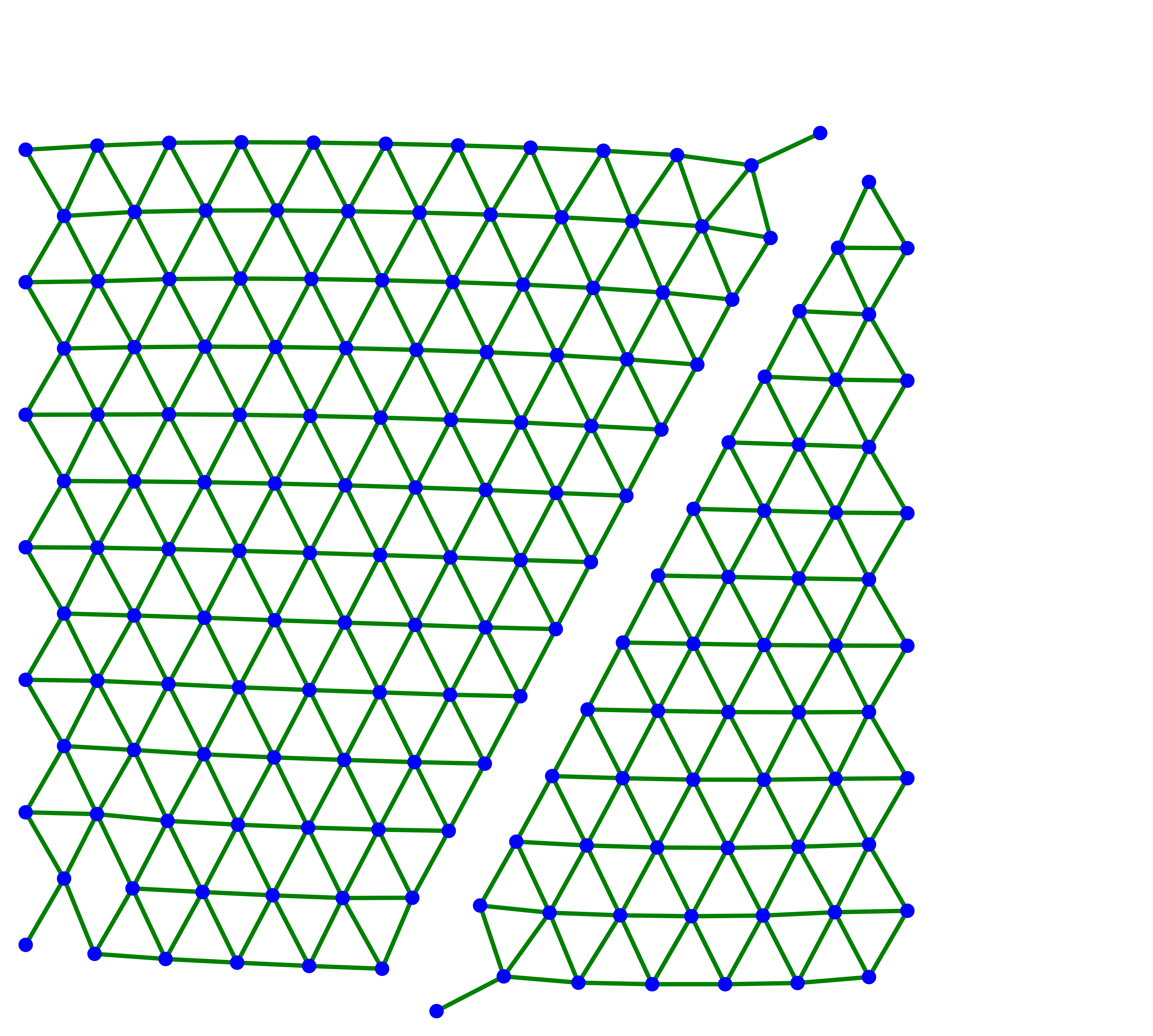}
      \caption{Step 180}
      \label{fig:square_diag_8}
  \end{subfigure}
  \hfill
  \begin{subfigure}[b]{0.298\textwidth}
      \centering
      \includegraphics[width=\textwidth]{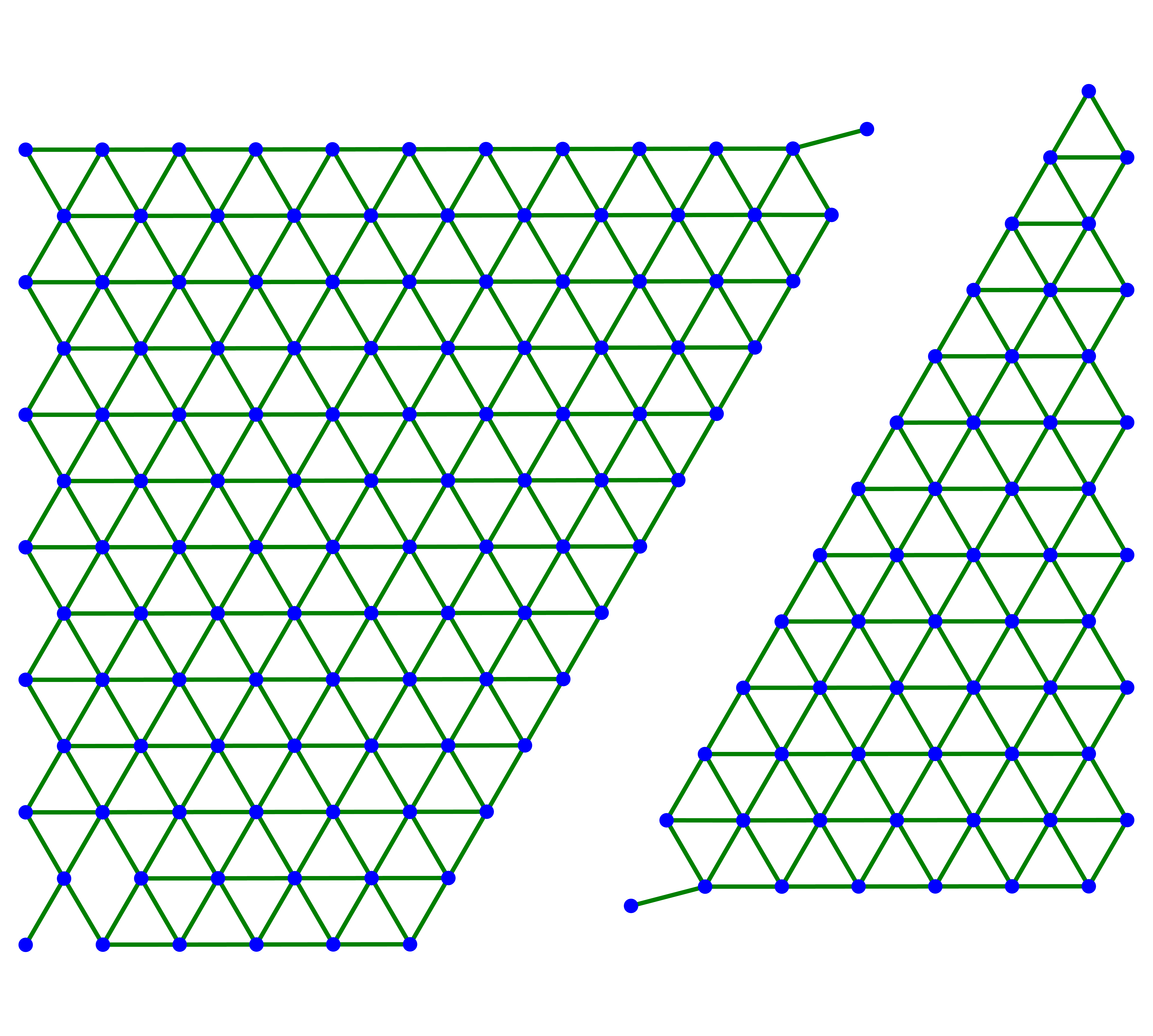}
      \caption{Step 225}
      \label{fig:square_diag_9}
  \end{subfigure} \\
     \caption{Evolution for sinusoidal  boundary conditions with angle $\pi/8$.}
     \label{fig:square_diag} 
     \end{figure}

\begin{figure}[h!]
  \centering
      \begin{subfigure}[b]{0.298\textwidth}
      \centering
      \includegraphics[width=\textwidth]{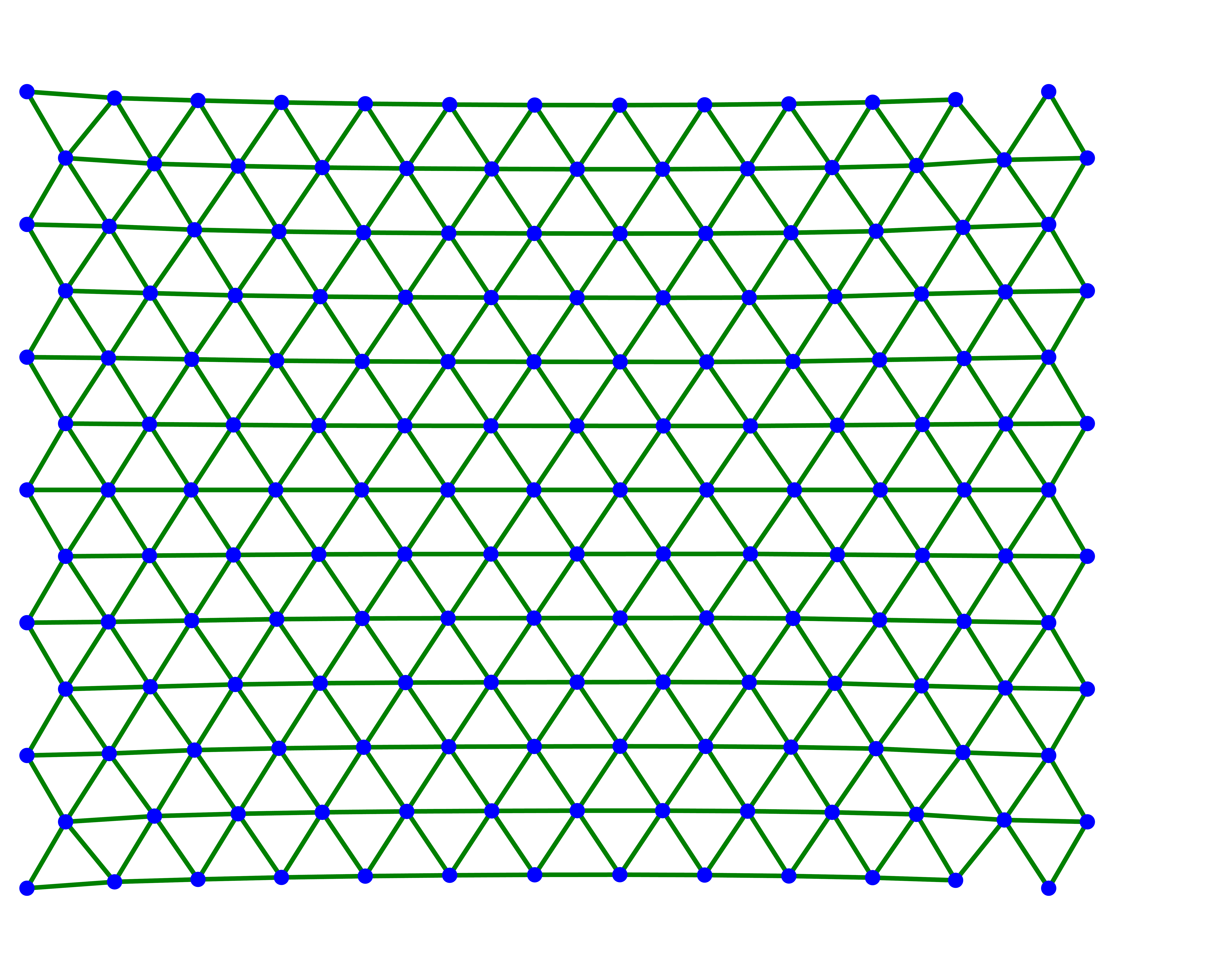}
      \caption{Step 15}
      \label{fig:square_2}
  \end{subfigure}
  \hfill
  \begin{subfigure}[b]{0.298\textwidth}
      \centering
      \includegraphics[width=\textwidth]{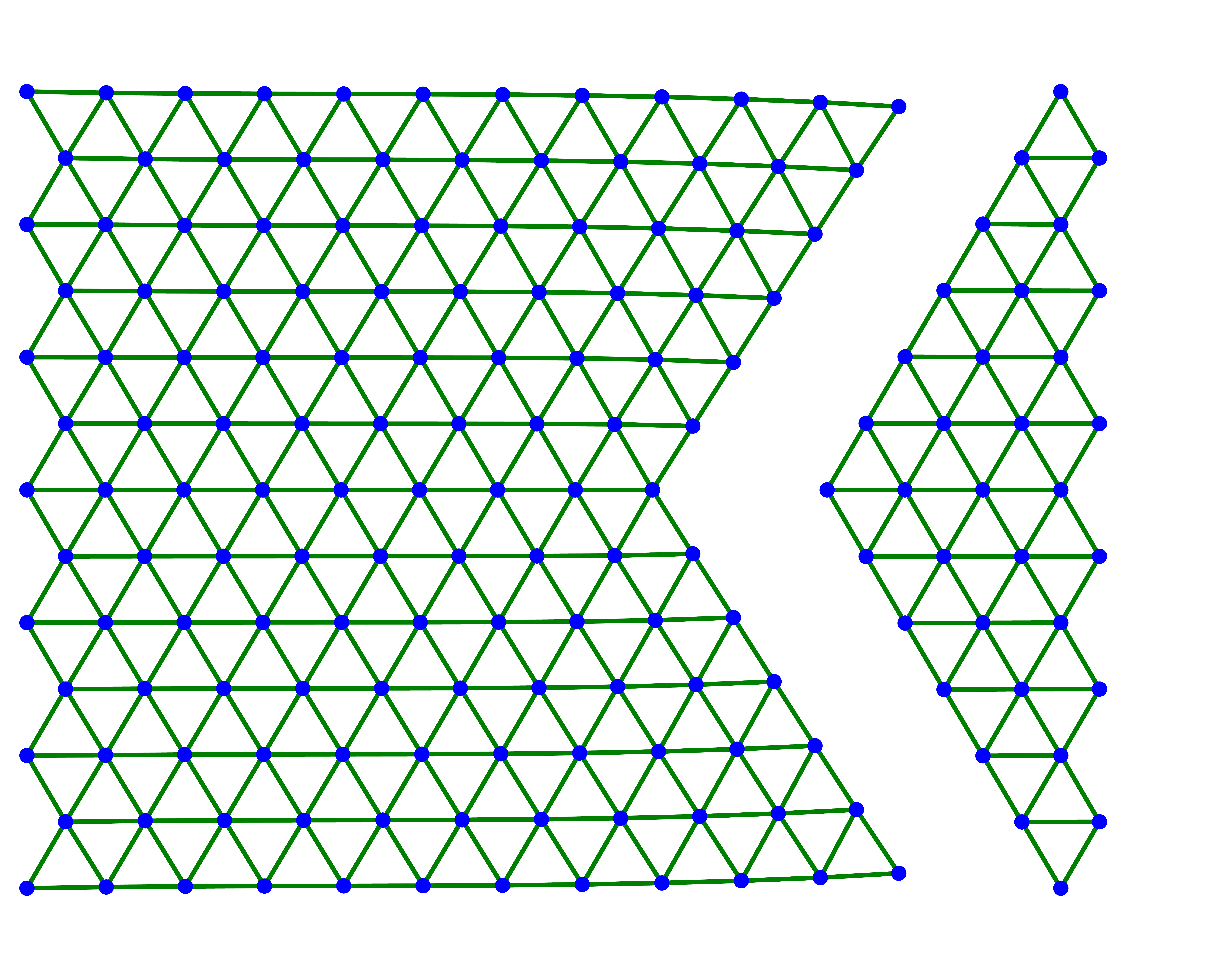}
      \caption{Step 17}
      \label{fig:square_3}
  \end{subfigure}
  \hfill
  \begin{subfigure}[b]{0.298\textwidth}
      \centering
      \includegraphics[width=\textwidth]{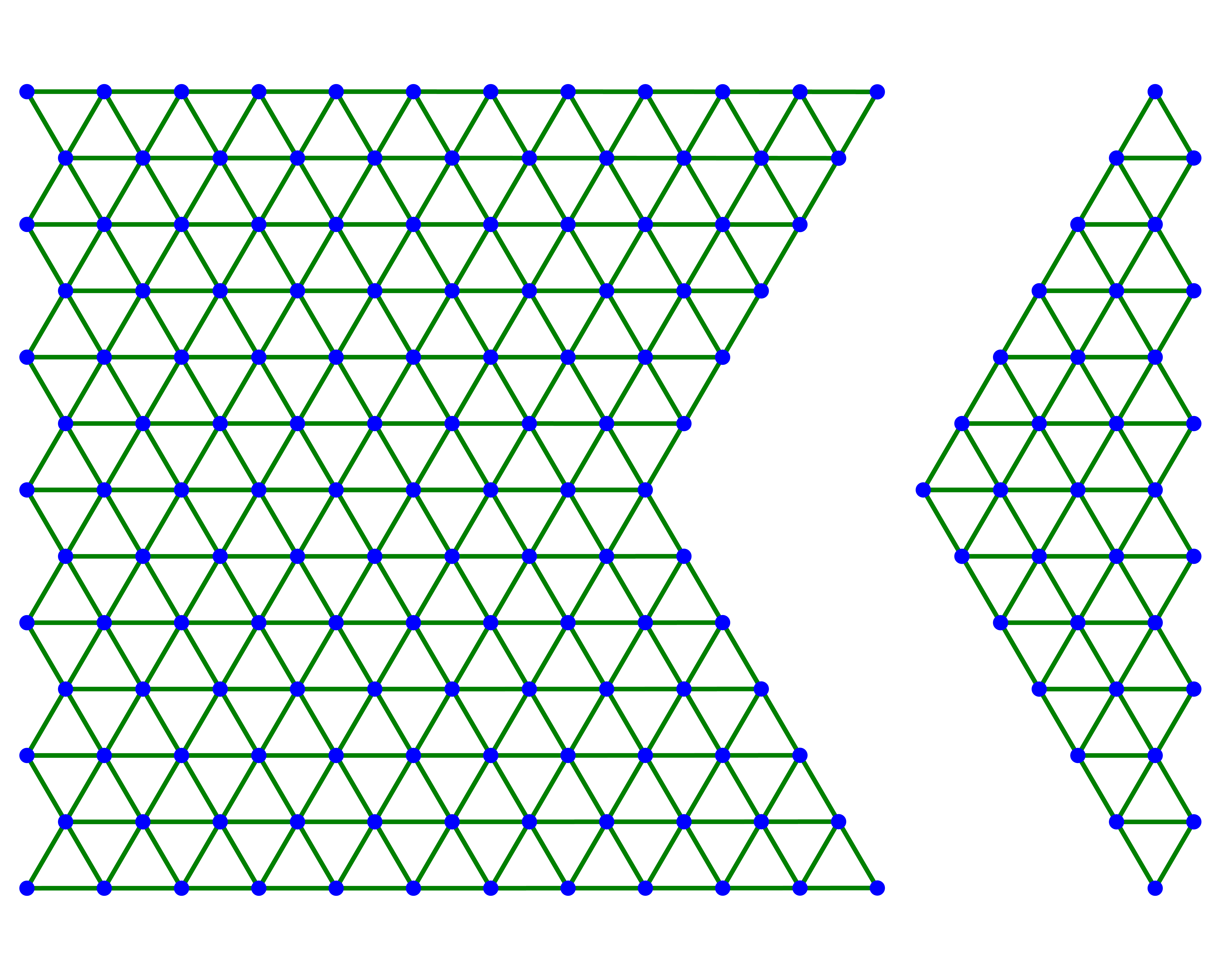}
      \caption{Step 45}
      \label{fig:square_4}
  \end{subfigure}
     \caption{Evolution for horizontal  stretching.}
     \label{fig:square_diag2} 
\end{figure}

We continue by showing several examples in   two dimensions. 
In all cases, we consider the triangular lattice  with  nearest neighbors and next-to-nearest neighbors  as described  in \eqref{eq: NNI}--\eqref{eq: NNIXX} with $\eta = 0.25$.  The evolution starts with a finite amount of atoms  arranged in a lattice   of equilateral triangles. The left-most atoms are fixed along the evolution while the right-most atoms are stretched horizontally or with respect to some angle. 

Figure \ref{fig:square_extr}  shows the case of a large value $\nu = 1$ for horizontal stretching.
Figures \ref{fig:square_l2_1}--\ref{fig:square_l2_3} depict three steps of the scheme with $L^2$-dissipation, where the right-most layer of atoms detaches from the material. Figures \ref{fig:square_kv_1}--\ref{sub@fig:square_kv_3} shows steps arising from a Kelvin-Voigt-type dissipation.
In this case,  the material does not develop a  sharp  crack interface. Instead, starting from the center more and more horizontal  bonds  are broken. This suggests that $\nu$ should be chosen much smaller.

  Figure \ref{fig:square_diag}  shows several steps of the respective minimizing movement scheme  with $L^2$-dissipation and small dissipation coefficient $\nu = 0.01$. In contrast to Figure \ref{fig:square_extr},  we use  sinusoidal  boundary conditions that move the right-most points of the lattice upwards in direction  $(\cos(\pi/8),\sin(\pi/8))$. As predicted  in \cite{FS153, FS152} for the static case, the lattice breaks apart along a crystallographic line,  see Figures  \ref{fig:square_diag_6} and \ref{fig:square_diag_7}. The two components are stopped from colliding into each other in the subsequent compression phase (see Figure  \ref{fig:square_diag_8}).  After stretching once more, the components  separate from each  other   reaching the final state in Figure  \ref{fig:square_diag_9}.  Although not shown here, the evolution arising from a dissipation of Kelvin-Voigt type does essentially not differ  from the one shown in Figure \ref{fig:square_diag}. Furthermore, Figure \ref{fig:square_diag2}  shows an experiment with boundary conditions in horizontal direction. Due to the symmetry of the lattice, the crack line can be kinked, in accordance to  \cite{FS153, FS152}.

Eventually, in Figure \ref{fig:stress-strain} we plot the stress-strain curve that corresponds to the stretching phase in a tension-compression experiment similar to the one shown in Figure \ref{fig:square_diag}.  More precisely, we computed the total stress  $|\sum_{i=1}^N\nabla_{y_i} \mathcal{E}(y(t); (y(s))_{s<t})|$, suitably normalized by the number of boundary interactions.  As the material is stretched, the stress increases almost linearly until the formation of a crack where the stress drops to zero. The flattening of the curve for higher strains is due to the nonlinearity of the interaction potentials. In  case of a smaller crack-threshold $R$, the material breaks already at a smaller loading. Accordingly,  the flattening of the curve is less pronounced   and the maximal value of the stress, which corresponds to  the transition from the  elastic to  the fracture phase, is smaller.

\EEE

\begin{figure}[h]
  \centering
  \begin{subfigure}[b]{0.49\textwidth}
    \centering
    \includegraphics[width=\textwidth]{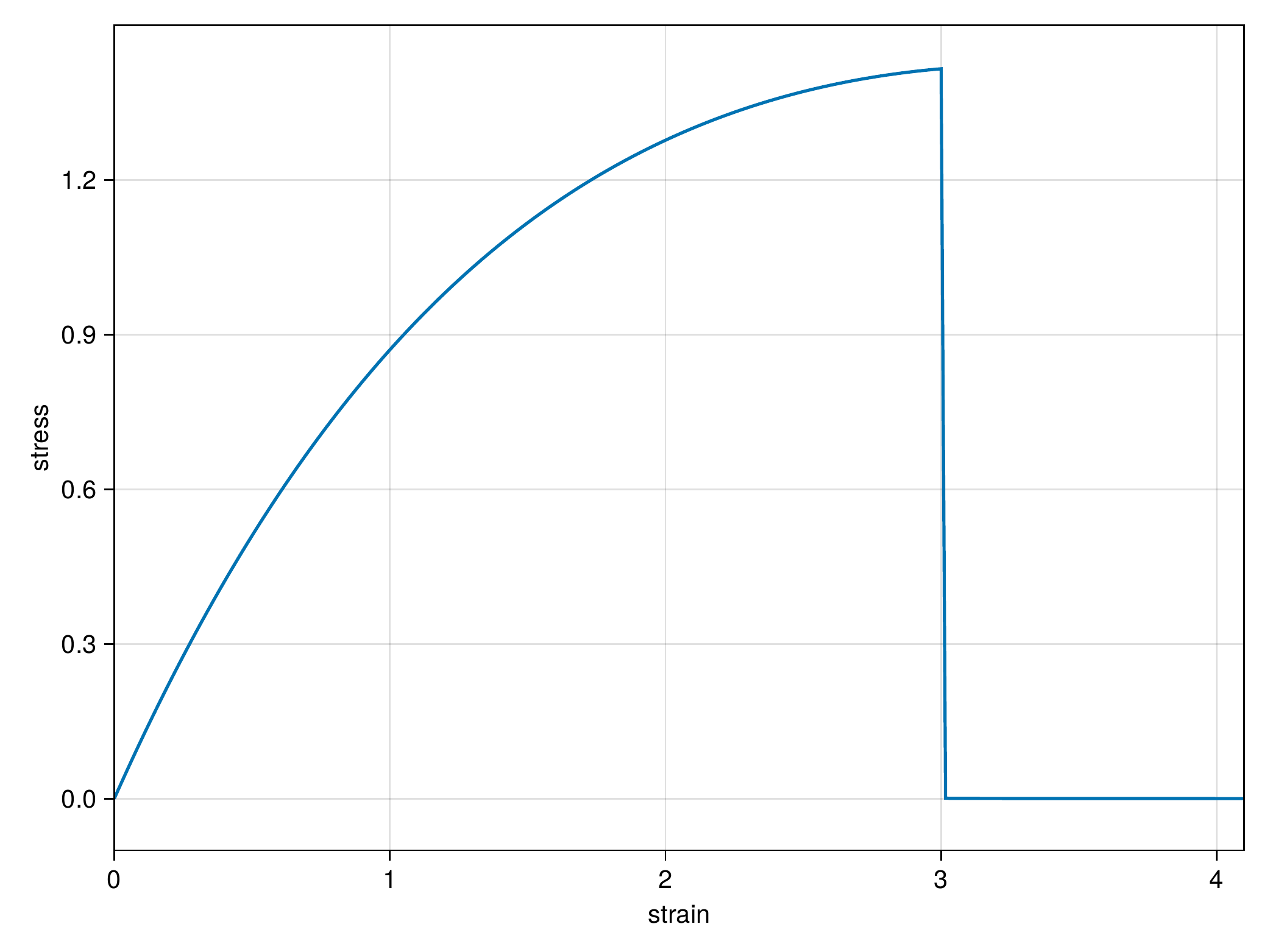}
    \label{fig:stress-strain1}
     \caption{threshold $R=1.2\eps$}
  \end{subfigure}
  \begin{subfigure}[b]{0.49\textwidth}
    \centering
    \includegraphics[width=\textwidth]{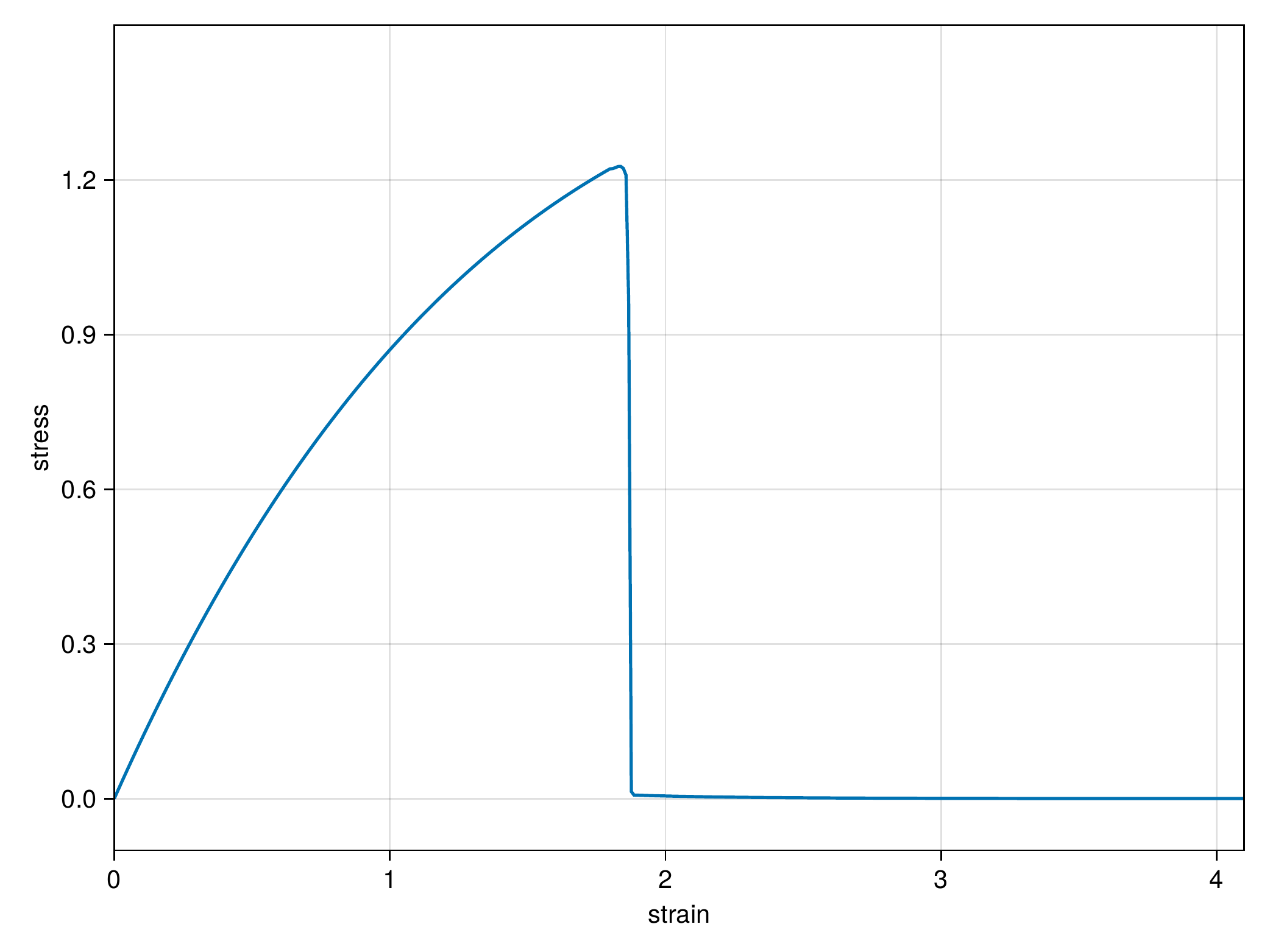}
    \label{fig:stress-strain2}
     \caption{threshold $R=1.07\eps$}
  \end{subfigure}
  \caption{Stress-Strain curves for different thresholds.}
  \label{fig:stress-strain} 
\end{figure}

\section{Proofs}\label{sec: proofs}

\subsection{Incremental minimization scheme}
 
  In this subsection, we analyze the scheme \eqref{eq: MM} for a given family of deformations $(y^\tau_j)_{j< k}=(y^\tau_j)_{j\leq k-1}$. For convenience, we denote the memory variable $M_{x,x'}((y^\tau_j)_{j< k})$ simply by $M^{x,x'}_{k-1}$. In a similar fashion, we express the  energy for fixed $x,x'\in \mathcal{L}(\Omega), x \neq x'$,  by  
 \begin{align}\label{eq: crazy1}
 E^{x,x'}_{k-1}(\cdot):=E_{x,x'}(\;\cdot\;;(y^\tau_j)_{j < k}),
 \end{align}
  and for  the sum over all contributions we write 
  \begin{align}\label{eq: crazy2}
  \mathcal{E}_{k-1}(\cdot):=\mathcal{E}(\;\cdot\;; (y^\tau_j)_{j < k})\,.
  \end{align}  
 Furthermore, in this subsection we drop the superscript $\tau$, i.e., we write $y_k$ and $t_k$ in place of $y_k^{\tau}$ and $t_k^\tau$, respectively. \EEE By convention, in proofs  we often drop the index $x,x'$ for   $W_{x,x'}$, $\varphi_{x,x'}$,  $ E^{x,x'}_{k-1}$, \EEE  $M^{x,x'}_{k-1}$, $R^{1}_{x,x'}$,   $R^{2}_{x,x'}$ \EEE if no confusion arises.

  We begin by proving a lemma that follows from the very definition of the   potential in \eqref{eq: memory2} and the incremental minimization scheme described in \eqref{eq: MM}.   
\begin{lemma}\label{lift_lemma}
    Given a time step $t_k$ and deformations $(y_j)_{j < k}$ at previous time steps, let $y_{k}$ be defined as in \eqref{eq: MM}. Then,
    \[\mathcal{E}_{k-1}(y_{k}) =\mathcal{E}_{k}(y_{k}).\] 
\end{lemma}
\begin{proof}
    It suffices to show that $E^{x,x'}_{k-1}(y_{k})=E^{x,x'}_k(y_{k})$  for all $x,x'\in \mathcal{L}(\Omega)$, $x \neq x'$. For simplicity, we drop  the superscript $x,x'$  in the proof.  Recall that by definition \eqref{eq: memory2} we have 
    \begin{align}\label{eq: Ek-1}
    E_{k-1}(y_{k}):=(1- \varphi(M_{k-1})) W ( |y_{k}(x') - y_{k}(x)| )  + \varphi(  M_{k-1}) \big(W(  M_{k-1}) \vee  W(|y_{k}(x') - y_{k}(x)| )\big)\,,
    \end{align}
    where $\varphi(M_{k-1})=0$ for $M_{k-1}\leq R^{1}$ and $\varphi(M_{k-1})=1$ for $M_{k-1}\geq R^{2}$. Note that by the  definition of the memory variable  in \EEE \eqref{eq: memory} we obtain 
    \[M_{k}= M_{k-1}\vee |y_{k}(x')-y_{k}(x)|\,.\]
    \emph{Case 1:} We first consider pairs $x,x'\in \mathcal{L}(\Omega)$, $x \neq x'$, that fulfill $M_{k-1}=M_{k}$. Then  $E_{k-1}(y_{k}) = E_{k}(y_{k})$  trivially holds. \\
    \emph{Case 2:} Let $x,x'\in \mathcal{L}(\Omega)$, $x \neq x'$, be such that 
    $M_{k-1}<M_{k}$. Then, we have 
    \[M_{k}=|y_k(x)-y_k(x')|= M_{k-1} \vee |y_k(x)-y_k(x')|\,,\] and since $W$ is increasing on  $[|x-x'|, \infty)$ by Assumption  \ref{W_decr_incr} \EEE and  $M_{k-1} \geq |x-x'|$ we get 
    \[W(M_{k-1})\vee W(|y_{k}(x') - y_{k}(x)|) = W(|y_{k}(x') - y_{k}(x)| )= W(M_{k})\vee W(|y_{k}(x') - y_{k}(x)|) \,.\] 
    Hence, by \eqref{eq: Ek-1} we obtain
    \[E_{k-1}(y_k)=W(|y_{k}(x') - y_{k}(x)| )=E_{k}(y_k)\,, \] which proves the assertion. 
\end{proof}


As a next step, we prove an estimate that will enable us to control the time derivative of the deformation.   As a preparation, consider  $g_i = g_i^\tau = g(t_i^\tau)$ for every $i \in \{0, \dots, T/\tau\}$.
Then, using $g \in H^1([0, T]; \R^{Nn})$, the fundamental theorem of calculus, and Jensen's inequality it follows that \EEE
\begin{equation}\label{g:control_1}
  \sum_{i=1}^{T/\tau}\frac{|g_{i}-g_{i-1}|^2}{ \tau}
  = \sum_{i=1}^{T/\tau} \tau \left|
    \frac{1}{\tau} \int_{(i-1)\tau}^{i \tau} \partial_t g(s) \, {\rm d}s
  \right|^2 \EEE
  \leq \int_{0}^{T} |\partial_t  g \EEE (s)|^2\, {\rm d}s
  =\|\partial_{t}  g \EEE\|^{2}_{L^2( [0,T]; \EEE \R^{Nn})}\,,\end{equation}
  and with Hölder's inequality
\begin{equation}\label{g:control_2} 
  \sum_{i=1}^{T/\tau} |g_{i}-g_{i-1}|
   \le \sqrt{\tau} \left(
    \sum_{i = 1}^{T/\tau} \frac{|g_i - g_{i-1}|^2}{\tau}
  \right)^{\frac{1}{2}} \sqrt{T / \tau}
  \le \sqrt{T} \|\partial_t g \|_{L^2(  [0,T]; \EEE  \R^{Nn})}.
\end{equation}
We thus have the  estimates
\begin{align}\label{eq: apriori}
\sum_{i=1}^{T/\tau} \frac{|g_{i}-g_{i-1}|^2}{\tau} \le C, \qquad \sum_{i=1}^{T/\tau}|g_{i}-g_{i-1}| \le C  \,,
\end{align}  
 for a constant $C$ independent of $\tau$.  \EEE
\begin{lemma}\label{lemma: main lemma}
There exists $C >0$ such that for  each $j \in \lbrace 0, \ldots,  T /\tau\rbrace $ it holds that 
\begin{align}\label{eq: newlemma1}  
   \mathcal{E}_{j}(y_{j}) + \sum_{i=1}^j \frac{\nu|y_{i}-y_{i-1}|^2}{2\tau} \leq  C \sum_{i=1}^j|g_{i}-g_{i-1}| + \sum_{i=1}^j \frac{\nu|g_{i}-g_{i-1}|^2}{2\tau} + \mathcal{E}_{0}(y_{0}).
\end{align}
\end{lemma}

\begin{proof}
We prove the statement by induction. The case $j= 0$ is trivial. Suppose that the estimate holds for $0 \le j  \le k-1$. In particular,  by \eqref{eq: apriori} we get \EEE
\begin{align}\label{eq: energy bound}
   \mathcal{E}_{j}(y_{j}) \le \hat{C} \quad \text{ for $0 \le j \le k-1$}
\end{align}
for some $\hat{C}>0$ depending on $g$  and $y_0$ \EEE but independent of $j$. By choosing $y_{k-1}+g_{k}-g_{k-1}$   as a test function in the minimization problem  \eqref{eq: MM} we get
\begin{equation}\label{testing}
     \mathcal{E}_{k-1} \EEE (y_{k})+\frac{\nu|y_{k}-y_{k-1}|^2}{2\tau}\leq  \mathcal{E}_{k-1} \EEE (y_{k-1}+g_{k}-g_{k-1})+ \frac{\nu|g_{k}-g_{k-1}|^2}{2\tau}\,. \\
\end{equation}
We aim at estimating $\mathcal{E}_{k-1}(y_{k-1}+g_{k}-g_{k-1})$. As before, this can be reduced to considering each $x,x' \in \mathcal{L}(\Omega)$, $x \neq x'$. First, by Assumption  \ref{W_zero} \EEE and \eqref{eq: energy bound} we get 
\begin{align}\label{eq: lower bound}
|y_{k-1}(x) - y_{k-1}(x')| \ge \bar{c}_{x,x'}
\end{align}
 for a constant $\bar{c}_{x,x'} >0$ depending on $W_{x,x'}$,  $g$, and $y_0$, \EEE but independent of $k$. By the regularity of $W_{x,x'}$ and Assumption  \ref{W_inf} \EEE we find that $W_{x,x'}$ is Lipschitz continuous on $[\bar{c}_{x,x'}/2,\infty)$. Therefore, also $y \mapsto W_{x,x'}(|y|)$ is Lipschitz continuous on $\lbrace y\in  \R^n \EEE \colon |y| \ge \bar{c}_{x,x'}/2 \rbrace$. We denote the Lipschitz constant by $L_{x,x'}$ which is independent of $k$.

We again drop subscripts and write    $E_{k-1}(y_{k-1}+g_{k}-g_{k-1})$, $W$, $L$, and $\bar{c}$ for convenience.   Note that for $z_1,z_2 \in  \R^n \EEE$ with $|z_1| \ge \bar{c}$ and $|z_2| \le \bar{c}/2$ we have $|W(z_1 + z_2) - W(z_1)| \le L |z_2|$.  Let  $z_1 := y_{k-1}(x)-y_{k-1}(x')$ and  $z_2 := (g_{k}-g_{k-1})(x)-(g_{k}-g_{k-1})(x')$. Observe that  $|z_1| \ge \bar{c}$ by \eqref{eq: lower bound},  and that $|z_2| \le \bar{c}/2$  for $\tau$ small enough since  \eqref{eq: apriori} implies  \EEE 
$$|(g_{k}-g_{k-1})(x)-(g_{k}-g_{k-1})(x')| \le |(g_{k}-g_{k-1})(x)| + |(g_{k}-g_{k-1})(x')|  \le \sqrt{2}|g_{k}-g_{k-1}| \le C\sqrt{\tau}.$$
 Then, we get   
\begin{align} \label{Lipschitz1}
\big|W(|y_{k-1}(x)-y_{k-1}(x')  +  &  (g_{k}-g_{k-1})(x)-(g_{k}-g_{k-1})(x')|) - W(|y_{k-1}(x)-y_{k-1}(x')|) \big| \notag \\ &  \le L| (g_{k}-g_{k-1})(x)-(g_{k}-g_{k-1})(x')| \notag \\
& \le  L \bigl(|g_{k}(x)-g_{k-1}(x)|+|g_{k}(x')-g_{k-1}(x')|\bigr)\,.
\end{align}
As $(a+b)\vee c \leq (a \vee c) + b$ for $a,b,c \ge 0$, we can also infer that 
\begin{equation}\label{Lipschitz_MM}
    \begin{aligned}
        &W(|(y_{k-1}+g_{k}-g_{k-1})(x)-(y_{k-1}+g_{k}-g_{k-1})(x')| ) \vee W(M_{k-1})\leq \\ &  \bigl(W(|y_{k-1}(x)-y_{k-1}(x')|)\vee W(M_{k-1})\bigr) \EEE + L \big(|g_{k}(x)-g_{k-1}(x)|+|g_{k}(x')-g_{k-1}(x')|\big)\,.
\end{aligned}
\end{equation} 
Putting \eqref{Lipschitz1} and \eqref{Lipschitz_MM} together we derive for the energy contribution ${E}_{k-1}(y_{k-1}+g_{k}-g_{k-1}) $ (see definition \eqref{eq: Ek-1}) the estimate
\begin{equation}\label{est: potential}
    {E}_{k-1}(y_{k-1}+g_{k}-g_{k-1})\leq {E}_{k-1}(y_{k-1})+ L \big(|g_{k}(x)-g_{k-1}(x)|+|g_{k}(x')-g_{k-1}(x')|\big)\,.
\end{equation}
From this estimate, summing over all pairs $x,x'$, we conclude 
\begin{align*}
    &\mathcal{E}_{k-1}(y_{k-1}+g_{k}-g_{k-1}) = \frac{1}{2}\sum_{\underset{x \neq x'}{x, x' \in \mathcal{L}(\Omega)}} {E}^{x,x'}_{k-1}(y_{k-1}+g_{k}-g_{k-1})\\ \leq \;&\frac{1}{2}\sum_{\underset{x \neq x'}{x, x' \in \mathcal{L}(\Omega)}} {E}^{x,x'}_{k-1}(y_{k-1})+ \frac{1}{2}\sum_{\underset{x \neq x'}{x, x' \in \mathcal{L}(\Omega)}} L_{x,x'} \, \big(|g_{k}(x)-g_{k-1}(x)|+|g_{k}(x')-g_{k-1}(x')|\big) \\  \leq \;&\mathcal{E}_{k-1}(y_{k-1}) +  \bar{L}(N-1) \sum_{x \in \mathcal{L}(\Omega)}  |g_{k}(x)-g_{k-1}(x)|,
\end{align*}
where $\bar{L}$ is the maximum of the involved Lipschitz constants. Eventually, we get  by Hölder's inequality \EEE
\begin{equation}\label{est: energy}
 \mathcal{E}_{k-1}(y_{k-1}+g_{k}-g_{k-1}) \le \mathcal{E}_{k-1}(y_{k-1}) +    \bar{L}(N-1)   \sqrt{N} \EEE |g_{k}-g_{k-1}|. 
\end{equation}
Note that $\bar{L}$ can be chosen independently \EEE of $k$  as the constants $\bar{c}_{x,x'}$ in \eqref{eq: lower bound} are independent of $k$. \EEE  Putting \eqref{testing} and \eqref{est: energy} together and rearranging terms yields   
\begin{equation*}
    \mathcal{E}_{k-1}(y_{k})-\mathcal{E}_{k-1}(y_{k-1})+\frac{\nu|y_{k}-y_{k-1}|^2}{2\tau}\leq   C |g_{k}-g_{k-1}| + \frac{\nu|g_{k}-g_{k-1}|^2}{2\tau}
\end{equation*} 
for some $C>0$ large enough. 
By Lemma \ref{lift_lemma} this becomes
\begin{equation*}
    \mathcal{E}_k(y_{k})-\mathcal{E}_{k-1}(y_{k-1})+\frac{\nu|y_{k}-y_{k-1}|^2}{2\tau}\leq   C |g_{k}-g_{k-1}| + \frac{\nu|g_{k}-g_{k-1}|^2}{2\tau}.
\end{equation*} 
This inequality added to \eqref{eq: newlemma1}   for $j=k-1$ gives the desired estimate \eqref{eq: newlemma1}  for $j= k$.  
\end{proof}

%

%
%
%

\subsection{Compactness}

In this subsection, we prove Theorem \ref{th: compactness}. For this reason, we again include the time step $\tau$ in our notation.  After the preparation in Lemma \ref{lemma: main lemma} the proof is standard. We  perform it here, \EEE however, for convenience of the reader.

\begin{proof}[Proof of Theorem \ref{th: compactness}]
 The \EEE right-hand side of \eqref{eq: newlemma1} for $j= T /\tau$  can \EEE be expressed  in terms of  $g $ as in \EEE \eqref{g:control_1}--\eqref{g:control_2}.  Thus,  from \eqref{eq: newlemma1}  \EEE we obtain
$$  \sum_{i=1}^{T/\tau} \frac{\nu|y^\tau_{i}-y^\tau_{i-1}|^2}{2\tau} \leq  C\sqrt{T}  \|\partial_{t} g\|_{L^2([0,T]; \R^{Nn})} + \frac{1}{2}\|\partial_{t} g\|_{L^2([0,T]; \R^{Nn})}^{2} + \mathcal{E}_{0}(y^\tau_{0}) -  \mathcal{E}_{T/\tau}(y^\tau_{T/\tau}) \le \hat{C} $$
for some $\hat{C}>0$, where in the last step we used the regularity of $g$ and the fact that $\mathcal{E}_{T/\tau}$ is bounded from below. 
    Considering the interpolation $\hat{y}_{\tau}$ as introduced before Theorem \ref{th: compactness}, we derive
        \[ \frac{\nu}{2}\int_{0}^{T}|\partial_{t}\hat{y}_{\tau}|^2\, {\rm d}s =   \sum_{i=1}^{T/\tau}\frac{\nu|y^\tau_{i}-y^\tau_{i-1}|^2}{2\tau} \le \hat{C}\,.\]
Therefore,  $\partial_{t}\hat{y}_{\tau}$ is uniformly bounded in $L^2([0,T];\mathbb{R}^{Nn})$. By weak compactness we have, up to a subsequence, 
    \[\hat{y}_{\tau}\rightharpoonup y \quad \text{in} \; H^1([0,T];\R^{Nn}) \quad \text{for some} \;y\in H^1([0,T];\R^{Nn})\,.\]      
  In particular, $y$ is H\"older-continuous    by the well-known inclusion $H^1(0,T)\subset C^{0,1/2}(0,T)$ and the convergence is uniform. In fact, by the fundamental theorem of calculus  we get
    \[|\hat{y}_{\tau}(t')-\hat{y}_{\tau}(t)|= \Big|\int_{t}^{t'} \partial_{t}y_{\tau}(s)\, {\rm d}s\Big|\leq \int_{t}^{t'} |\partial_{t}y_{\tau}(s)| \, {\rm d}s \,,\]
    and using  Hölder's inequality we  deduce 
    \[|\hat{y}_{\tau}(t')-\hat{y}_{\tau}(t)|\leq \Bigl(\int_{t}^{t'} |\partial_{t}y_{\tau}(s)|^2 \, {\rm d}s\Bigr)^{1/2}  \Bigl(\int_{t}^{t'}1\, ds\Bigr)^{1/2}\leq C |t'-t|^{1/2}\,.\] 
 By Arzel\`a-Ascoli, $\hat{y}_{\tau}$ converges uniformly to the limit function $y$.
  For $k \in \{1, \ldots, T / \tau\}$ and $t \in ((k-1)\tau, k\tau]$ we define the piecewise affine interpolation $\hat g_\tau(t) \defas g_{k-1}^\tau + \tau^{-1} (t - (k-1)\tau) (g_k^\tau - g_{k-1}^\tau)$.
 Then, by the fundamental theorem of calculus and Hölder's inequality, for any $t \in [0, T]$ it follows  
 \begin{equation*}
    |\hat g_\tau(t) - g(t)| \leq |\hat g_\tau(t) - g((k-1)\tau)| + |g((k-1)\tau) - g(t)| \leq 2 \|\partial_t g\|_{L^2([0, T]; \R^{Nn})} \sqrt{\tau},
 \end{equation*}
where $k$ is the smallest natural number  such that  $k\tau \geq t$. \EEE 
This shows $\hat g_\tau \to g$ in $L^\infty([0, T]; \R^{Nn})$.
Note also that we have $\hat{g}_{\tau}(t)=\hat{y}_\tau(t)$ in $\mathcal{L}_{D}(\Omega)$ for all $t \in [0,T]$ \EEE since by construction $y^\tau_k=g^\tau_k$ on $\mathcal{L}_{D}(\Omega)$ for all $k$. This yields $y(t) \in \mathcal{A}(g(t))$ for all $t \in [0,T]$.
 It remains to observe that also the piecewise constant interpolation $y_{\tau}$ converges uniformly to $y$. Indeed, we know that for every $t\in[0,T]$ we have $y_{\tau}(t)=\hat{y}_{\tau}(s)$ for a certain $s=k\tau$ with $|s-t|\leq \tau$. In particular, we get
  \[ |\hat{y}_{\tau}(t)-y_{\tau}(t) |= |\hat{y}_{\tau}(t)-\hat{y}_{\tau}(s)|\leq C \sqrt{|t-s|}\leq C \sqrt{\tau} \,.\]
This shows that also $y_{\tau}$ converges uniformly to $y$  (along the same subsequence).  
\end{proof}

\subsection{The limiting  delay differential equation}

Next,  we   consider the limiting evolution $y$ identified in Theorem \ref{th: compactness}, and  give the proof of Theorem \ref{th: main result}. We first introduce some additional notation. Recalling \eqref{eq: crazy1}--\eqref{eq: crazy2} we let 
\[E^{x,x'}_{\tau}(t, \cdot):=E^{x,x'}_{k}(\cdot)=E_{x,x'}(\,\cdot\,; (y_j)_{j\leq k}) \quad \text{and}\quad \mathcal{E}_{\tau}(t, \cdot):=\mathcal{E}_{k}(\cdot)=\mathcal{E}(\cdot; (y_j)_{j\leq k})\,\]  and $M^{x,x'}_{\tau}(t):=M^{x,x'}_k$ for $t\in (\tau k, \tau(k+1)]$. 
Although the potentials $W_{x,x'}$ are smooth, the function $v\mapsto E^{x,x'}_{\tau}(t,v)$ is possibly not differentiable if  $W(M^{x,x'}_{\tau}(t)) = W_{x,x'}(|v(x)-v(x')|)$. \EEE Therefore, we need to introduce the concept of subdifferential for $\lambda$-convex functions, see  e.g.\ \cite{sant-am}. As before, for convenience we drop $x,x'$ in the notation.  Recall that 
$$ E_{\tau}(t, v) = (1- \varphi(M_\tau(t)))  \, W ( |v(x') - v(x)| )  + \varphi(  M_\tau(t)) \big(W(  M_\tau(t)) \vee  W(|v(x') - v(x)| )\big) $$
for $t\in  [0,T] \EEE $, see \eqref{eq: Ek-1}.  
 In fact, by Assumption  \ref{W_inf} \EEE one observes that $E_{\tau}(t, \cdot)$ is $\lambda$-convex for some $\lambda <0$, i.e., $v\mapsto E_{\tau}(t, v) -\frac{\lambda}{2}|v|^2$ is convex.  In this case, the subdifferential is defined as  
\begin{align}\label{eq: subdiffi}
\partial E_{\tau}(t,v):=\{p\in \mathbb{R}^{\bar{N}n}: E_{\tau}(t,y)\geq  E_{\tau}(t,v) +p \cdot (\bar{y}-\bar{v})  + \tfrac{\lambda}{2}|\bar{y}-\bar{v}|^2 \ \text{for all $y \in \R^{Nn}$} \}\,,
\end{align}
where $\bar{y} \in \R^{\bar{N}n}$ and $\bar{v} \in \R^{\bar{N}n}$ denote the vectors of  images of  $x \in \mathcal{L}(\Omega) \setminus \mathcal{L}_D(\Omega)$ under $y$ and $v$, respectively. \EEE 
 By the theory of subdifferentials we compute that  
\[  \partial E_{\tau}(t,v) =   \begin{cases}
   \lbrace (1-\varphi(M_{\tau}({t})))  W'(|v(x)-v(x')|) z_{x,x'}  \rbrace & W(M_{\tau}({t})) > W(|v(x)-v(x')|) \\
  \lbrace   W'(|v(x)-v(x')|) z_{x,x'} \rbrace  &  W(M_{\tau}({t})) < W(|v(x)-v(x')|)  \\ 
     [1-\varphi(M_{\tau}({t}))   , 1 ] \,  \EEE W'(|v(x)-v(x')|)z_{x,x'}   &  W(M_{\tau}({t})) = W(|v(x)-v(x')|)
\end{cases}\,,\]
 where  $z_{x,x'} := (\nabla_{v_i}  |v(x)-v(x')| )_{i=1,\ldots,\bar{N}} \in \R^{\bar{N}n}$ whose entries are given by $\nabla_{v_i}   |v(x)-v(x')| = \frac{v(x)-v(x')}{|v(x)-v(x')|}  \in \R^n $ if $x=x_i$,  $\nabla_{v_i}   |v(x)-v(x')| = \frac{v(x')-v(x)}{|v(x)-v(x')|}$ if $x'=x_i$, and $\nabla_{v_i}   |v(x)-v(x')| =0 $ if $x_i \neq x,x'$.   Here, we use the notation $W' = \frac{\rm d}{{\rm d} r} W$. \EEE Note that for  $W(M_{\tau}({t})) \neq W(|v(x)-v(x')|)$, the function is differentiable and therefore the subdifferential only consists of one element, whereas in the case $W(M_{\tau}({t})) = W(|v(x)-v(x')|)$ it is given by an interval.  If $\partial E_{\tau}(t,v)$ is not single valued, we denote by $\partial^\circ E_{\tau}(t,v)$ the element of minimal norm. In our case, this is 
\begin{align}\label{eq: subdiff1}
\partial^\circ E_{\tau}(t,v)=   \begin{cases}
    (1-\varphi(M_{\tau}({t})))  W'(|v(x)-v(x')|) z_{x,x'}  & W(M_{\tau}({t})) \ge W(|v(x)-v(x')|) \\
   W'(|v(x)-v(x')|) z_{x,x'}    &  W(M_{\tau}({t})) < W(|v(x)-v(x')|)  
\end{cases}\,.
\end{align}
For the entire energy, this reads as
\begin{align}\label{eq: subdiff2}
\partial \mathcal{E}_{\tau}(t,v)=\frac{1}{2}\sum_{\underset{x \neq x'}{x, x' \in \mathcal{L}(\Omega)}}\partial E^{x,x'}_\tau(t,v), \quad \partial^\circ\mathcal{E}_{\tau}(t,v)=\frac{1}{2}\sum_{\underset{x \neq x'}{x, x' \in \mathcal{L}(\Omega)}}\partial^\circ E^{x,x'}_\tau(t,v)
\end{align}
 for $t\in  [0,T]\EEE$, where as before $\partial^\circ \mathcal{E}_{\tau}(t,v)$ denotes the element of minimal norm.  After these preparations, we proceed with the proof of Theorem \ref{th: main result}.

\begin{proof}[Proof of Theorem \ref{th: main result}]
Since by construction  $y^\tau_{k+1}\in \mathrm{arg} \min_{v} \mathcal{E}_{k}(v)+\frac{\nu|v-y_{k}|^2}{2\tau}$, for times $t\in (k\tau,(k+1)\tau)$  we obtain 
\begin{align}\label{eq: tolimit}
  \nu \EEE \partial_{t} \hat{y}_{\tau}(t) \in - \partial \mathcal{E}_{\tau} (t,y_\tau(t)) \EEE \quad \text{on $\mathcal{L}(\Omega) \setminus \mathcal{L}_{D}(\Omega)$}\,,
 \end{align}
 where $\partial$ is again taken with respect to the variables $y_1,\ldots, y_{\bar{N}}$, cf.\ \eqref{eq: subdiffi}. \EEE   We aim to pass to the limit in this equation. The term on the left-hand side converges weakly to $ \nu \EEE\partial_{t}y$ by Theorem \ref{th: compactness}. We thus consider the limit of the right-hand side.

Because of the uniform convergence $y_{\tau}\to y$, we particularly know that for an arbitrary pair of lattice points $(x,x')$ we have  \[|y_{\tau}(t,x')-y_{\tau}(t,x)|\to |y(t,x)-y(t,x')| \quad \text{ as $\tau \to 0$ for all $t\in [0,T]$\EEE}\,.\] 
 With the continuity of $W'_{x,x'}$, we thus obtain 
\begin{align}\label{eq: conv1}
W'_{x,x'}(|y_{\tau}(t, x)-y_{\tau}(t, x')|)\to W'_{x,x'}(|y(t, x)-y(t, x')|)\quad \text{ for all } t\in [0,T]\,. 
\end{align}
Due to the fact that the convergence is uniform in time, we also get the continuity of $M_{\tau}(t)$ with respect to $\tau$, i.e.,  we get \EEE 
\begin{align}\label{eq: conv2}
M^{x,x'}_{\tau}(t)=\sup_{s\leq t-\tau }|y_{\tau}(s,x)-y_{\tau}(s,x')| \to  \sup_{s< t} |y(s,x)-y(s,x')|
\end{align}     
as $\tau \to 0$ for each pair $(x,x')$ and all $t \in [0,T]$. With the definition in \eqref{eq: memory}, this means
\begin{align}\label{eq: cont in M}
M^{x,x'}_{\tau}(t)  \to M_{x,x'}( (y(s))_{s<t} \EEE ) =: M^{x,x'}(t). 
\end{align}
 The convergences \EEE \eqref{eq: conv1}--\eqref{eq: conv2} and   the continuity of $\varphi$ allow us to go to the limit in $\partial E^{x,x'}_\tau(t, y_\tau(t))$, i.e.,
\begin{align}\label{eq: limiting passage}
\lim_{\tau \to 0}  \partial E^{x,x'}_\tau(t, y_\tau(t)) =   \partial{E}_{x,x'}(y(t); (y(s))_{s<t}) \EEE 
\end{align}
for all $(x,x')$ and all $t \in [0,T]$, where
\[  \partial{E}_{x,x'}(y(t); (y(s))_{s<t}) \EEE =   \begin{cases}
   \lbrace (1-\varphi(M^{x,x'}(t)))  W'_{x,x'}(|y(t,x)-y(t,x')|) z_{x,x'}  \rbrace  \\
  \lbrace   W'_{x,x'}(|y(t,x)-y(t,x')|) z_{x,x'} \rbrace  \\ 
    [1-\varphi(M^{x,x'}(t))   , 1 ] \EEE \, W'_{x,x'}(|y(t,x)-y(t,x')|) z_{x,x'},  
\end{cases}\]
depending on whether $W_{x,x'}(M^{x,x'}(t))$ is bigger, smaller, or equal to $ W_{x,x'}(|y(t,x)-y(t,x')|)$.  Summing over all pairs $(x,x')$ and recalling the definition of the energy in \eqref{eq: main energy} we conclude  
 $${\lim_{\tau \to 0}   \partial \mathcal{E}_{\tau}(t, y_\tau(t))  =  \partial \EEE  \mathcal{E}(y(t); (y(s))_{s<t})   \quad \text{on $\mathcal{L}(\Omega) \setminus \mathcal{L}_{D}(\Omega)$}.}$$
Now, passing to the limit in \eqref{eq: tolimit} we conclude
$$  \nu \EEE \partial_{t} {y}(t) \in - \partial \mathcal{E}(y(t);  (y(s))_{s<t} ) \quad \text{on $\mathcal{L}(\Omega) \setminus \mathcal{L}_{D}(\Omega)$} $$
for almost every $t \in [0,T]$. Recalling the argumentation in \cite[Proposition 2.2]{sant-am} we deduce that  
$$  \nu \EEE \partial_{t} {y}(t) = - \partial^\circ \mathcal{E}(y(t); (y(s))_{s<t}) \quad \text{on $\mathcal{L}(\Omega) \setminus \mathcal{L}_{D}(\Omega)$} $$
for almost every $t \in [0,T]$.  This concludes the proof. 
\end{proof}

\begin{remark}[Role of $\varphi$ and viscosity $\nu$]\label{rem: discussion}
{\normalfont
 The delicate part in the above proof is the limiting passage $\tau \to 0$ in $ \partial \mathcal{E}_{\tau}(t, \cdot)$. In particular, the continuity of $\varphi$ is essential in  \eqref{eq: limiting passage}: \EEE if $\varphi$ was not continuous instead,  as in the possible modeling assumption $\varphi(r) = 0 $ for $r \le R^1$ and  $\varphi(r) = 1 $ for $r > R^1$ discussed in Subsection \ref{sec: evo model}, \EEE it would not be clear how to obtain  convergence \EEE as $\tau \to 0$. We also mention that we consider a model of rate type with viscosity in order to control $\partial_t y$ and thus to obtain uniform convergence. This is crucial in order to obtain continuity of the memory variable, see \eqref{eq: cont in M}. For a model in the realm of rate-independent processes, uniform convergence cannot be guaranteed, and thus we possibly do not have continuity in \eqref{eq: cont in M}.
}
\end{remark}

\begin{remark}[Proof of existence for \eqref{eq: new model}]\label{rem: adapt}
{\normalfont
We briefly indicate the necessary adaptations in the proof in order to derive existence of solutions to  \eqref{eq: new model}. First, by an argument analogous to the one in Lemma \ref{lemma: main lemma} one can show that
\begin{align}\label{in the new proof}
\sum_{k=1}^{T /\tau} \frac{\nu|\bar{\nabla }y_{k}- \bar{\nabla} y_{k-1}|^2}{2\tau} \le C.
\end{align}
By the triangle inequality we find for each $j \in \lbrace 1,\ldots N\rbrace$ that
$$|y_k(x_j) - y_{k-1}(x_j)| \le |y_k(x_1) - y_{k-1}(x_1)| + \sum_{i=2}^j  |(y_k(x_i) - y_{k-1}(x_i)) -  (y_k(x_{i-1}) - y_{k-1}(x_{i-1})) |. $$
By a discrete H\"older's inequality we then find
$$ |y_k(x_j) - y_{k-1}(x_j)|^2 \le Cj \Big( |y_k(x_1) - y_{k-1}(x_1)|^2 + \sum_{i=2}^j  |(y_k(x_i) - y_{k-1}(x_i)) -  (y_k(x_{i-1}) - y_{k-1}(x_{i-1})) |^2 \Big). $$
Summation over all atoms and using $y_k(x_1) = g_k(x_1)$ for all time steps  yields \EEE
\begin{align*} \sum\nolimits_{j=1}^N|y_k(x_j) - y_{k-1}(x_j)|^2 &  \le CN^2 \big( |g_k(x_1) - g_{k-1}(x_1)|^2  \vspace{-0.2cm} \\ &  \ \ \ \ \ +   \sum\nolimits_{i=2}^N  |(y_k(x_i) - y_{k-1}(x_i)) -  (y_k(x_{i-1}) - y_{k-1}(x_{i-1})) |^2 \big)   \\
& \le CN^2\big(|g_k(x_1) - g_{k-1}(x_1)|^2 + \eps^2 |\bar{\nabla }y_{k}- \bar{\nabla} y_{k-1}|^2\big) ,
\end{align*}
where the last step follows from the definition in \eqref{eq: dg}. This along with \eqref{in the new proof} as well as \eqref{eq: apriori} shows 
$$ \sum_{k=1}^{T /\tau} \frac{|y_k  - y_{k-1}|^2}{2\tau} = \sum_{k=1}^{T /\tau} \sum_{j=1}^N \frac{|y_k(x_j)  - y_{k-1}(x_j)|^2}{2\tau} \le C,
$$ 
and at this point the compactness proof in Theorem \ref{th: compactness} can be repeated. Concerning the passage $\tau \to 0$, we only need to replace the inclusion \eqref{eq: tolimit} by 
$$  \nu \EEE \bar{\rm D} \partial_{t} \hat{y}_{\tau}(t) \in - \partial \mathcal{E}_{\tau}(t, y_\tau(t))\,,$$
 where $\bar{\rm D}$ is defined as in \eqref{def: VektorD}, 
and then we can pass to the limit on both sides as before. 
}
\end{remark}

\begin{remark}[Extended model \eqref{eq: alternative energy}]\label{rem: new}
{\normalfont

We briefly comment on the fact that the proof also works for the enhanced energy \eqref{eq: alternative energy}. 
 In fact, the continuity of $\varphi$ and $W_{x,x',x''}$ 
as well as the continuity of the time-discrete memory variable $M^{\tau}_{x,x'}$ (see \eqref{eq: cont in M}), which follows from the uniform convergence $y_{\tau}\to y$, ensure the validity of \eqref{eq: limiting passage} also in this setting.
} \end{remark}
\EEE

\section{Conclusion and future work}
In this paper, we have developed a mathematically sound framework for the evolutionary description of brittle fracture on the atomistic level. The proposed model features an energy that comprises a memory variable for each pair interaction accounting for the irreversibility of the fracture process. We have proven the existence of a quasi-static evolution driven by this energy resulting in a delay-differential equation.  Furthermore, our analysis is supplemented with numerical experiments in 1D and 2D. Here, we observed a crack evolving along crystallographic lines, in accordance with previous experimental and mathematical results. \\
As mentioned before, this study forms the foundation of a forthcoming work \cite{FriedrichSeutter} where we will investigate the atomistic-to-continuum passage for crack evolution. Additionally, it would be of interest to further examine possible extensions of the model to non-local and multi-body energies. In regard to numerics, one could  consider the effects of different lattice structures on the crack evolution.   

\EEE

\section*{Acknowledgements} 
This research was funded by the Deutsche Forschungsgemeinschaft (DFG, German Research Foundation) - 377472739/GRK 2423/1-2019. The authors are very grateful for this support. This work was supported by the DFG project FR 4083/3-1 and by the Deutsche Forschungsgemeinschaft (DFG, German Research Foundation) under Germany's Excellence Strategy EXC 2044 -390685587, Mathematics M\"unster: Dynamics--Geometry--Structure.
\EEE


\end{document}